\documentclass[10pt,reqno]{amsart} 
\usepackage{mathtools}
\usepackage{color}
\usepackage{graphicx}
\usepackage{multirow}
\usepackage{booktabs}
\usepackage{subfigure}

\usepackage[hidelinks]{hyperref}
\usepackage{lineno}
\theoremstyle{definition}
\newtheorem{thm}{Theorem}[section]

\newtheorem{lem}[thm]{Lemma}

\newtheorem{rem}[thm]{Remark}
\newtheorem{exm}[thm]{Example}

\newtheorem{algo}[thm]{Algorithm}
\numberwithin{equation}{section}

\newcommand{\abs}[1]{\left\vert#1\right\vert}

\allowdisplaybreaks
\begin{document}
	
	\title[A simple and fast FDM for  the variable-order fractional Laplacian]{A simple and fast finite difference method for  the integral fractional  Laplacian of  variable order}

	\author{Zhaopeng Hao}
	\address{School of Mathematics, Southeast University, Nanjing, China}
	\email{ zphao@seu.edu.cn; zhaopenghao2015@gmail.com}
	
	\author{Siyuan Shi}
	\address{School of Mathematics, Southeast University, Nanjing, China}
	\email{ mshisiyuan@163.com}

\author{Zhongqiang Zhang}
\address{Department of Mathematical Sciences, Worcester Polytechnic Institute, Worcester, Ma 01609 USA }
\email{zzhang7@wpi.edu}
 
	\author{Rui Du}
	\address{School of Mathematics, Southeast University, Nanjing, China}
	\email{(Corresponding author)   rdu@seu.edu.cn}

	\subjclass[2010]{ 35B65, 65M70, 41A25, 26A33}

	\date{\today}

	
	\keywords{Nonlocal Laplacian, Pseudo-differential operator, Approximation Properties, Convergence, Fast solver }
	\date{\today}
	\subjclass[2010]{ 65N06, 65N12, 65T50, 35R11, 26A33}
	
	\begin{abstract}

For the fractional Laplacian of variable order, an efficient and accurate numerical evaluation in multi-dimension  is a challenge for the nature of a singular integral. We propose a simple and easy-to-implement finite difference scheme for the multi-dimensional variable-order fractional Laplacian defined by a hypersingular integral. We    prove that the scheme is of  second-order convergence and apply the developed finite difference scheme to solve various  equations with the variable-order fractional Laplacian. We  present a fast solver with quasi-linear complexity of the scheme for computing variable-order fractional Laplacian and corresponding PDEs. Several numerical examples  demonstrate  the accuracy and efficiency of our  algorithm and verify our theory.

		
	\end{abstract}

	\maketitle
	
%
 
	
	\section{Introduction}


During the past few decades, fractional calculus has been extensively explored as a tool for developing more sophisticated but computationally tractable models. They can accurately describe complex physical phenomena, manifesting in long-range and nonlocal interactions, self-similar structures, sharp interfaces, and memory effects, which can not be adequately captured by  their integer counterpart.  
One of the most important fractional calculus tools is the fractional   Laplacian  of the  constant $\alpha$-th order,  defined 
as an singular integral \cite{bucur2016nonlocal} 
\begin{equation} 
(-\Delta)^{\alpha/2}u(x) :=  c_{d, \alpha} \text{P.V.} \int_{\mathbb{R}^d } \frac{u(x) - u(y)} {|x-y|^{d+\alpha}} \, \text{d}y ,\quad c_{d,\alpha } := \frac{2^{\alpha -1} \alpha \Gamma \big( (\alpha+d)/2 \big)}{\pi^{d/2} \Gamma \big( 1-\alpha/2 \big)} \label{def-c-int-frac-lap},
\end{equation} 
 with a normalized constant $c_{d,\alpha}$, 
 has been intensively investigated in the literature in the past two decades.   In particular, various applications have been found for the nonlocal operators  \eqref{def-c-int-frac-lap} such as fracture mechanics, image denoising, and phase field models in which jumps across lower-dimensional subsets and sharp transitions across interfaces (see \cite{Gatto-Hesthaven2015,Gunzburger-18,Laskin2000,Woyczynski2001}). 
 

 \subsection{Motivation for  the variable-order  model}
%
The constant-order operator \eqref{def-c-int-frac-lap} may be less sufficient for  the heterogeneous effect due to the spatial variability of complex medium.  To account  for  heterogeneity,  the variable-order operators depending on the spatial location variable  have been alternatively proposed. In this work, we consider the following  $\alpha(x)$-th  order fractional Laplacian
	\begin{equation}
		(-\Delta)^{\alpha(x)/2} u(x) :=  c_{d,\alpha(x)} \text{P.V.}  \int_{\mathbb{R}^d} \frac{ u(x)-u(y)} {|x-y|^{d+\alpha(x)}}\text{d}y,   \quad c_{d,\alpha (x)} := \frac{2^{\alpha(x) -1} \alpha(x) \Gamma \big( (\alpha(x)+d)/2 \big)}{\pi^{d/2} \Gamma \big( 1-\alpha(x)/2 \big)}.  \label{def-v-int-frac-lap}
	\end{equation}
 
 The variable-order extension  variable order  enables us to truly capture the non-smooth effects such as fractures by prescribing variable degree of smoothness across the scales. For example, the authors in \cite{Antil-Rautenberg-2019} advocated the variable-order operator to  perform regularization in image processing. They used large value of the order $\alpha(x)$ in the flat region and the smaller value in the neighborhood of  edges. 
In \cite{Giusti-2020,Darve-DElia-Garrappa-2022}, the authors pointed out that the usage of the variable order is partially motivated  from the theoretical study of galaxy rotation curves and  the dark matter.  Their study requires the field equation for the gravitational potential to become a variable-order fractional Poisson’s equation.
Another example is from the groundwater flow. The authors in \cite{Liu-Sun-Zhang-2019} justified that  the super-diffusion depends on spatial variables due to the local variation of aquifer properties. We also refer the interested readers to other  applications including  biology \cite{farquhar2018computational}, geophysics \cite{MuHWZ2021},  spatial statistics   \cite{Ruiz-MedinaAA2004} and so on. 
%

When restricted to  the bounded domain, 
equations
 with the fractional Laplacian of constant order as the leading operator usually have non-physical weak boundary singularity (boundary layer) unless boundary conditions are special. 
The  boundary layer is often undesirable for the real world applications. As a remedy, the authors in \cite{Zheng-Wang-2020} suggested to use the Caputo-type variable-order operator and analyzed a one-dimensional (1D) two-point boundary value problem. 
 %
 In Section \ref{sec-num-experm}, we use a variable-order fractional Laplacian  requiring variable order $\alpha(x)=2$ along the boundary and observe in  Example \ref{Ex-known-solution}  
 that a full convergence order can be achieved on uniform meshes, which may suggest full regularity of solutions to such equations. 
 This observation may lead to  more practical fractional calculus in real applications.

\subsection{Related definitions of variable-order fractional Laplacian}
  Changing directly the constant-order  into the variable-order may increase not only the model capability but also  the complexity in the computation.  
  For example, one can  use  the variable-order function  $\alpha (x,y)$ in place of $\alpha(x)$ in \eqref{def-v-int-frac-lap}, which depends on both the spatial variable $x$ and the integration  variable $y$ in the singular integral in the recent work \cite{Marta-Glusa-2022,Ok2023}.   
  The wellposedness of corresponding fractional Poisson equations or other partial differential equations   can be straightforwardly established from the variational formulation. However,  the definition  does not  allow a simple Fourier transform of the fractional Laplacian as in the constant-order fractional Laplacian and Fourier-transformed based computational methods  are not readily applicable.

  The authors  in \cite{Darve-DElia-Garrappa-2022}  used the Fourier transform to extend the variable-order fractional Laplacian.  
  However, the  Fourier-type definition therein 
  (see also Remark \ref{rem:variable-fourier-v}) is different from the one considered in this work.  
  Here we prove  that the singular integral type definition \eqref{def-v-int-frac-lap} can be rewritten via Fourier transforms  but is different from the definition in \cite{Darve-DElia-Garrappa-2022}; see Theorem \ref{prop-equivalence-definitions} for the details.  One  desirable feature of their definition is that the  kernel is convolution-type but its explicit form is not available and have to be calculated through the forward and backward Fourier transforms.    
  
  Except for the above commonly used approaches to  define the variable-order  fractional Laplacian, one can also consider the spectral definition (see  \cite{YuZZZ-2022}) and harmonic-extension type definition (see \cite{Antil-Rautenberg-2019}). For the spectral definition,  the associated computation is straightforward when the eigen-values or eigen-functions are explicitly available.   In this paper, we will not go through every definition  but limit our attention into the singular-integral type \eqref{def-v-int-frac-lap}.     

 \subsection{Literature review and research gap}
 The close or explicit form of the $\alpha(x)$-th order fractional Laplacian of  functions seldom exists.  In this work, we aim for an efficient numerical evaluation
	for the fractional Laplacian of variable order \eqref{def-v-int-frac-lap} with its applications in  numerical solutions to nonlocal PDEs. 
  In 1D case, we note that the authors in \cite{ZhuangLAT2009} showed the equivalence between the Fourier-type definition and two-sided Riemann-Liouville derivatives  reformulation. Following upon their work, finite difference methods and fast solvers  have been developed for the  fractional Laplacian of coordinate-dependent type (see  \cite{Pang-Sun-2021,WangSLL2023}).  However, to the best of our knowledge,  
	we are not aware of any numerical results  in  2D/3D  in the literature due to the complexity of the operator discretization. 
  Even if one can extend the finite difference scheme such as the popular Grunwald approximation to compute 1D variable-order fractional operators, the rigorous convergence analysis is not provided yet.  The current research aims to fill this computational and theoretical  gap.

 
 For the constant-order fractional Laplacian, some progress has been made and    extensive numerical methods have been developed among the computational mathematics community  particularly during the past around five years. 
 Finite or boundary  element methods are   developed in the context of the boundary integral equations, and they  can be straightforwardly extended to solve the integral equations involving the volume integral (see \cite{AnisworthG17}). However, it is extremely   difficult to implement the finite element methods (FEM)   particularly in 3D as it entails complicated numerical quadrature for the double integrals \cite{Feist-Bebendorf-2022}.  To avoid the expensive computation of the double integrals, meshless methods  or  collocation methods   have been developed. 
 For special domains such as a disk,  the accurate and efficient spectral method  using the eigen-functions was proposed in \cite{Xu-Darve-2020} and \cite{Hao-Li-Zhang-Zhang-2021}. For a general domain,  radial basis functions (RBF) methods using the standard Gaussian functions \cite{Burkard-Wu-Zhang-2021} and other generalized multi-quadratic functions have been developed. However,   RBF methods suffer from the notorious ill-condition issue which makes it harder for solving the large-scale size problem.    
 
 For the large-scale computation and easy implementation in multi-dimensions for fractional Laplacian of constant order, finite difference methods or collocation methods have been proposed recently. 
 Using the singularity subtraction technique, the authors in \cite{Minden-Ying-2020} proposed a simple solver via the Taylor expansion under the high-regularity assumption that the target function $u(x) \in C^6(\Omega)$ but the stability of the finite difference scheme is unknown. The authors in \cite {Duo-Zhang-2019} converted the fractional Laplacian with the strong singularity into the one with the weak singularity through the integration by parts, and then constructed a quadrature-based finite difference approximation. However, the stability estimates  were not provided and it is not clear if the associated finite difference scheme is stable  when applied to the numerical solution of PDEs involving the fractional Laplacian operator. The authors in \cite{Antil-Sinc-2021} presented the collocation method  based on the sinc basis functions  and rigorous analysis of the method was later provided in \cite{Antil-2023-Sinc}. 
  Nonetheless, the  analysis  therein relies on the properties of Fourier transform  and cannot  be directly extended to the problem considered in this work.


We remark that the Fourier transform of fractional Laplacian of constant-order allows designs of various and efficient numerical methods.  However, this is not the case for the variable-order fractional Laplacian as   the semi-group  and symmetry  properties cannot extend,  and   the explicit and elementary representation  of the Fourier transform of variable order does not exist.  Without these properties, straightforward extensions of previous working numerical methods are highly  nontrivial including various discretization methods such as spectral methods  in \cite{tang2018hermite,Tang-Wang-Yuan-Zhou-2020}, the fast finite element implementation \cite{Sheng-2023}, Dunford-Taylor reformulation based finite element methods and spectral methods and so on.

 		In   \cite{Hao-Zhang-Du-2021} we establish an easy-to-implement and efficient finite difference method based on the generating functions approximation theory. The computation of the coefficients or weights in  finite difference approximation can be  efficiently calculated through the fast discrete Fourier transform and a fast solver is implemented for the \textit{structured} resulting linear systems.  Along this line of research, we aim to extend the idea of simple finite difference discretization  in  \cite{Hao-Zhang-Du-2021} to the variable-order context, and  propose an efficient solver for the resulting \textit{unstructured} matrix induced by the non-convolution kernel in the variable-order operator \eqref{def-v-int-frac-lap}.

\subsection{Contributions and outline of this work}
 
The \textit{significance and novelty} of the paper is that it is the first work addressing a fast and accurate finite difference approximation  for the numerical evaluation of the  multi-dimensional fractional Laplacian of variable order \eqref{def-v-int-frac-lap}.  Our approximation is  robust and stable  in the sense that we can get a second-order approximation even for the nonsmooth piecewise variable-order function $\alpha(x)$ (see Tables \ref{Approx_VO_Lap_1D}, \ref{Approx_VO_Lap_2D} and \ref{Approx_VO_Lap_3D}).
 
 The approximation is derived from the discrete Fourier transform.  Although the explicit form of the Fourier transform for the variable-order case does not exist, we find that it can be defined through the Fourier transform as in the constant-order case.  Such equivalence allows us to discretize the singular operator in the frequency space which avoids the complicated quadrature. 
 
 Compared to the constant-order counterpart,  the variable-order definition in multi-dimensions uses the non-convolution kernel  and thus the corresponding discrete counterpart becomes non-Toeplitz, which brings up the extra computational difficulty regarding the efficient solver part. Nonetheless,  by the Fourier transformation and the low-rank approximation   of the Fourier symbol  $|\xi|^{\alpha(x)}$ (see Lemma \ref{lem-low-rank-approximation-fourier-symbol}), we are able to overcome the difficulty and construct  a fast solver  based on the observation that the unstructured resulting matrix can be approximately decomposed into the linear combination of the structured matrices. 
 
 Several numerical results are demonstrated to show  the accuracy and efficiency of the proposed scheme. Numerically, we  find that the appropriately  selected variable-order model can avoid the non-physical boundary singularity as in the constant-order model. 

In a nutshell, the main contributions of this work are described as follows. 
	\begin{itemize}
\item We  establish the equivalence and provide its proof  between Fourier-type definition and singular-integral form. 
\item  We  derive  a second-order  finite difference approximation for the  variable-order fractional Laplacian in multi-dimensions and present a rigorous convergence analysis. 
\item  We provide a fast solver with quasi-linear complexity  to efficiently compute the variable-order operator and related PDEs.  
		\item We apply the developed finite difference discretization to solve various fractional PDEs with variable-order fractional Laplacian and carry out the extensive numerical experiments to show the accuracy and efficiency of our theory and algorithm. 
	\end{itemize}

	We organize this paper as follows. In Section \ref{sec-FDM-approx}, we  present a fractional finite difference approximation based on the equivalence between the Fourier-type definition and singular integral form.  Moreover, we describe its implementation with  a quasilinear fast solver, and show  the accuracy and robustness of our approximation.   
  In Section \ref{sec-app-frac-ellip-eqn}, we apply the discrete approximation to the numerical solution for the nonlocal elliptic PDEs.   We carry out the extensive  numerical experiments and demonstrate the accuracy and efficiency of our scheme.  We  put the proofs of the theoretical results  in Section \ref{sec-proofs} for the interested readers, including  the second-order approximation and the convergence analysis for the  finite difference method in 1D case.  Finally, we conclude this work in Section \ref{sec-conclusion}.

	\section{ A second-order finite difference approximation }\label{sec-FDM-approx}

	In Subsection \ref{subsection-equvalence-with-fourier-transform}, we establish the equivalence   for  fractional-order Laplacian between singular integral form and  the Fourier transform. Such equivalent characterization provides us different perspectives and a starting point to design an efficient numerical method, the finite difference method via the semi-discrete Fourier transform which will be described in Subsection \ref{subsec-finite-difference-approx}. We describe the implementation for the calculation of the coefficients or weights associated with the approximation in Subsection \ref{subsection-impletation-coefficients} and the fast solver in Subsection \ref{subsec-implementation}. Finally we present a numerical example in Subsection \ref{subsection-numer-exeriments-approximation}. 

	\subsection{Variable-order fractional Laplacian defined via Fourier transform}\label{subsection-equvalence-with-fourier-transform}
 
	In our discussion, we assume that the variable-order function in the definition \eqref{def-v-int-frac-lap} is continuous and bounded,  that is $\alpha(x) \in C(\mathbb{R}^d)$,  and 
	\begin{equation}
		0<\alpha_{\min} \leq  \alpha (x)  \leq  \alpha_{\max} < 2. \label{assumption-alpha-function}
	\end{equation}
	
	The Fourier transform and its inverse are defined respectively as follows
	\begin{equation}
		\mathcal{F} [u] (\xi)= \int_{\mathbb{R}^d} e^{-\mathbf{i} \xi\cdot x } u(x)\text{d}x  \label{def-fourier-transform}	
	\end{equation}
	and 
	\begin{equation}
		\mathcal{F}^{-1} [ \hat{u}] (x)= \frac{1}{(2\pi)^d } \int_{\mathbb{R}^d} e^{\mathbf{i} \xi\cdot x } \hat{u}(\xi)\text{d}\xi.	\label{def-inv-fourier-transform}
	\end{equation}	
	where $ \mathbf{i}$ is the complex symbol $\mathbf{i}^2 =-1$. 
	Following the argument in the constant case as in \cite{bucur2016nonlocal}, 
	we show that the equivalence of  two definitions still holds in the variable-order context. 
	\begin{thm}[Equivalence of the definitions]\label{prop-equivalence-definitions}
		Let $u(x)\in C_b^2(\mathbb{R})$, which is the space of bounded and twice continuously differentiable functions.  The     $\alpha(x)$-th  variable-order fractional Laplacian is equivalent to the one defined in the Fourier transform, i.e., 
		\begin{equation}
			(-\Delta)^{\alpha(x)/2} u(x) = \mathcal{F}^{-1} \big[|\xi|^{\alpha(x)} \mathcal{F} [u] (\xi)\big](x), \label{def-v-frac-lap}
		\end{equation}
		when both sides of the equality  \eqref{def-v-frac-lap} are well-defined. 
	\end{thm}

 \begin{rem}\label{rem:variable-fourier-v}
 In \cite{Darve-DElia-Garrappa-2022}, the authors proposed the Fourier-type fractional Laplacian as follows: 
\begin{eqnarray}
(-\Delta)^{\alpha(x)/2} u(x) :=  \mathcal{F}^{-1} \bigg[  \frac{  1  }{\mathcal{F} [K] } \mathcal{F} [u]   (\xi)\bigg](x), \quad  K(x) := \frac{ c_{d,-\alpha(x)} } {|x|^{d-\alpha(x)}}.  \label{def-darve-fourier}
\end{eqnarray}
It is clear that the two Fourier-type definitions are different since  the inverse of the Fourier transform of kernel function $K$ is the univariate function in variable $\xi$ which is not equal to the  function $|\xi|^{\alpha(x)}$. Note that the definition \eqref{def-darve-fourier} leads to the convolution type operator while the definition \eqref{def-c-int-frac-lap} is the type of non-convolution. 
 \end{rem}

	The equivalence of the definitions inspires us to discretize the operator in the Fourier or frequency spaces rather than the direct discretization   in the physical space using the complicated finite difference-quadrature.

	
	\subsection{A second-order finite difference approximation}\label{subsec-finite-difference-approx}	
	
	Denote $h$ as the spatial stepsize of the grids.	Let  $h \mathbb{Z}^d$ be the set of infinite grids, i.e., 
	\begin{equation}
		h\mathbb{Z}^d=\{jh: j=(j_1,j_2,\cdots, j_d) \in \mathbb{Z}^d  \}. \label{def-grid-points-set}
	\end{equation}
	Throughout the paper, the summation index $j$ is multi-index, so is the index $k$. 
	
	For a function $u_h: h \mathbb{Z}^d \longrightarrow \mathbb{R}$, define the semi-discrete Fourier transform as 
	\begin{equation}
		\mathcal{F}_h [u_h] (\xi)= h^d \sum_{x\in h \mathbb{Z}^d } e^{- \mathbf{i}\xi \cdot x } u_h(x). \label{def-discrete-fouier-transform}
	\end{equation}	
	Denote the parameterized domain $D_h=\big[ -\frac{\pi}{h}, \frac{\pi}{h}\big]^d$. 	Then we have the discrete Fourier inversion formula
	\begin{equation} 
		u_h(x) = \mathcal{F}_h^{-1} \big[ \mathcal{F}_h [u_h] \big ] (x) =\frac{1}{(2\pi)^{d}} \int_{D_h} e^{\mathbf{i }\xi \cdot x } \mathcal{F}_h [u_h](\xi) \text{d}\xi. \label{def-inverse-discrete-transform}
	\end{equation}

	For a continuous function $u(x)$, let $u_h (x)= u (x)$ for $ x\in h\mathbb{Z}^d $. 	Inspired by the idea for 1D in \cite{huang2016finite} and our previous work \cite{Hao-Zhang-Du-2021} for multi-dimensional cases, we propose the discrete variable-order Laplacian operator as follows: 
	\begin{equation}
		(-\Delta_h)^{\alpha(x)/2} u(x) : = \mathcal{F}_h^{-1} \big[ M_h(\xi) ^{ \alpha (x) /2} \mathcal{F}_h [u_h] (\xi) \big](x)  \label{def-discrete-frac-lap}
	\end{equation}
	with  
	\begin{equation}
		M_h(\xi) := \sum_{p=1}^d \frac{4}{h^2} \sin^2 \big(\frac{\xi_p h}{2}\big). \label{def-multiplier}
	\end{equation}

Introduce the following spectral Barron space (see \cite{EMW2022, MengM2022})
	\begin{equation}
		\mathcal{B}^s(\mathbb{R}^d):= \{ u\in L^1(\mathbb{R}^d):  \int_{\mathbb{R}^d}( 1+|\xi|^{s}) |\mathcal{F}[u](\xi)| \text{d}\xi < \infty \}.
	\end{equation}
	and the induced norm $\|\cdot \|_{	\mathcal{B}^s(\mathbb{R}^d)} $. It is clear that  $\|u \|_{	\mathcal{B}^s(\mathbb{R}^d)} \leq c  \|u \|_{	\mathcal{B}^t(\mathbb{R}^d)}$ for $s<t$.

	For any grid function $v_h$ on $\Omega_h$, define the discrete maximum norm as $$ \|v_h\|_{L_h^\infty}= \max_{x\in \Omega_h} \{|v_h(x)|\}. $$
	
	We state the main theoretical result in this work.  
	\begin{thm}[Approximation properties]\label{thm-approximation-property}
		Suppose $\alpha (x)$ satisfy the assumption \eqref{assumption-alpha-function}. 
		Let $u \in \mathcal{B}^{s+\alpha_{\max}} (\mathbb{R}^d)$ with $s \leq 2$.  For the discrete fractional Laplacian operator defined in \eqref{def-discrete-frac-lap} and $x\in \Omega_h$, it holds that
		\begin{equation}
			\|(-\Delta)^{\alpha(x)/2} u(x)- (-\Delta _h)^{\alpha(x)/2} u_h(x)\|_{L_h^\infty} \leq c   h^{s} \|u\|_{ \mathcal{B}^{s+\alpha_{\max}}}.  \label{thm-convergence-order}
		\end{equation}
	\end{thm}
We postpone the proof of second-order approximation properties and put it in the section \ref{sec-proofs}. 
 \subsection{Implementation}\label{subsection-impletation-coefficients}
	
	In the practical computation, we want to numerically evaluate the operator at the grid points,  that is,  $x=x_j=jh\in \Omega_h$.  Thus, we have 
	\begin{eqnarray}
		(-\Delta_h)^{\alpha(x_j)/2} u(x_j) & =& \frac{1}{(2\pi)^d} \int_{D_h} e^{\mathbf{i} \xi x_j} M_h(\xi)^{\alpha(x_j)/2} \bigg( h^d\sum_{x_k\in \Omega_h}e^{- \mathbf{i}\xi  x_k} u(x_k) \bigg) \text{d}\xi \notag\\
		&=& h^d \sum_{x_k\in \Omega_h} \bigg( \frac{1}{(2\pi)^d} \int_{D_h} e^{- \mathbf{i} \xi \cdot (x_k-x_j)}  M_h(\xi)^{\alpha(x_j)/2}   \text{d}\xi \bigg) u(x_k) \notag\\
		&=&  \frac{1}{h^{\alpha(x_j)} } \sum_{x_k\in \Omega_h} a_{k-j}^{ (\alpha(x_j))} u(x_k) , \label{def-quasi-convolution-approximation}
	\end{eqnarray}	
	where the weights 
	\begin{eqnarray}
		a_{k-j}^{(\alpha(x_j))} :=	 \frac{1}{(2\pi)^d} \cdot h^{d+\alpha(x_j) } \int_{D_h} e^{- \mathbf{i} \xi \cdot (x_k-x_j)}  M_h(\xi)^{\alpha(x_j)/2}   \text{d}\xi .
	\end{eqnarray}
	Making a change of variable $\eta=h \xi$ leads to 
	\begin{eqnarray}
		a_{k-j}^{(\alpha(x_j))} =	 \frac{1 }{(2\pi)^d} \int_{[-\pi,\pi]^d} e^{- \mathbf{i}\eta \cdot (k-j)}  \bigg(\sum_{p=1}^d 4\sin\big (\frac{\eta_p}{2}\big)^2\bigg) ^{\alpha(x_j)/2}   \text{d}\eta. \label{def-a_k_alpha}
	\end{eqnarray}
	In the practical computation, for each fixed grid point $x_j$, we need to compute  the coefficients $	a_{n}^{(\alpha (x_j))}$ for $ |j_p| \leq N_p$ for $p=1,\dots,d$  (Here $n=k-j$ and $N_p$ denotes the number of grid points in $x_p$ direction. For the sake of simplicity, let us assume they all equals to $N$).

	To alleviate the burden of  tedious notations, we drop the dependence on the variable $x_j$  and abbreviate it as $a_{n}^{(\alpha_j )}$ without the confusion.

	In 1D, the analytical expression of the above integrals can be derived.  In fact,  recall the formula \cite[Page 399, 3.631.9]{gradshteyn2014table}
	\begin{eqnarray}
		\int_0^{\pi/2} \cos ^{\nu-1 }(  x)  \cos(ax) \text{d}x = \frac{\pi}{ 2^\nu \nu  B(\frac{\nu+ a+1}{2} , \frac{\nu-a+1}{2} )} ,\quad  
		\nu >0. \label{table-of-integrals-formula}
	\end{eqnarray}
	Here $B(\cdot, \cdot)$ denotes the Beta function. 
	Take $\nu=\alpha_j+1$, $a= 2n$ and make a change of variable $x= \pi/2-t$. Then the formula \eqref{table-of-integrals-formula} reduces to 
	\begin{eqnarray}
		(-1)^n	\int_0^{\pi/2} \sin ^{\alpha_j }(  x)  \cos(2nt) \text{d}t &=& \frac{\pi}{ 2^{\alpha_j+1} (\alpha_j+1)  B(\frac{\alpha_j+1+ 2n+1}{2} , \frac{\alpha_j+1-2n+1}{2} )} \notag\\ 
		&=& \frac{\pi \Gamma(\alpha_j+1)}{ 2^{\alpha_j+1} \Gamma( \frac{\alpha_j}{2} +n+1) \Gamma ( \frac{\alpha_j}{2} -n+1)}.
	\end{eqnarray}

	Due to the symmetry of the integrand, the coefficients \eqref{def-a_k_alpha} can be reduced to 
	\begin{eqnarray}
		a_{n}^{(\alpha_j)} &:=&	 \frac{2^{\alpha_j} }{\pi} \int_0^\pi    \sin^{\alpha_j}(\eta/2) \cos(k \eta  )   \text{d}\eta \notag\\
   &=& \frac{2^{\alpha_j+1}}{\pi} \int_{0}^{\pi/2} \sin ^{\alpha_j}( t)  \cos(2n t)  \text{d}t  \notag\\
		&= &\frac{  (-1)^n \Gamma( \alpha_j +1) } {\Gamma( \frac{\alpha_j}{2} +n+1) \Gamma ( \frac{\alpha_j}{2}-n+1)}.
	\end{eqnarray}
	
	In multi-dimension cases, since the explicit formula is difficult to derive, we will use the standard numerical integration to compute the integrals. Specifically,  take an integer number $M>\max\{ N_p, \, p=1,...d\}$ (here $N_p$ denote the number of grid points in $x_p$ direction) and step-size  $ \delta= 2\pi/M$. Denote  $\varphi(\eta)= [\sum_{p=1}^d 4\sin^2(\frac{\eta_p}{2}) ) ]^{\alpha_j/2}$. Applying the trapezoidal rule, we have 
		\begin{eqnarray}\label{numerical-quadrature-coefficients}
			a^{(\alpha_j)}_{n} &=& \frac{1}{(2\pi)^d}\int_{[-\pi,\pi]^d} \varphi(\eta) e^{-\mathbf{i} (\sum_{p=1}^dn_p\eta_p)}\text{d}\eta  \notag\\
			&\approx& \frac{1}{M^d} \sum_{m_1=0}^{M-1} \sum_{m_2=0}^{M-1} \cdots \sum_{m_d=0}^{M-1}\varphi(m_1\delta,m_2\delta,\cdots,m_d\delta) e^{-\mathbf{i} \sum_{p=1}^d(m_pn_p\delta)} \notag\\
   &:=&  	\tilde{a}^{(\alpha_j)}_{n}.
		\end{eqnarray}
		With the expression above we can use Matlab built-in function {`\sf{ifftn}'} to compute the coefficients 	$\tilde{a}^{(\alpha_j)}_{n}$ for $ 0\leq n_p\leq M-1$ efficiently with the computational cost $ \mathcal{O}(M^d\log^d M)$.  

		\subsection{ Fast computation of the discrete variable-order Laplacian } \label{subsec-implementation}
		
In this subsection we present a fast solver for the numerical evaluation of the fractional Laplacian.  In general, we are interested in the numerical evaluation of the target function $u(x)$ restricted in the bounded computational domain $\Omega$. Suppose that $\Omega$ is a square (or a rectangular domain if we take a nonisotropic different mesh size  in each spatial direction in the last subsection). If not, we can embed the general bounded domain $\Omega$  into a larger square (or rectangular) domain $\tilde{\Omega}$, consider the evaluation on the  domain $\tilde{\Omega}$ instead and do the restriction in the domain $\Omega$.   The following lemma is a key to 	develop a fast algorithm.

		\begin{lem}\label{lem-low-rank-approximation-fourier-symbol}
			For any tolerance $\epsilon>0$, there exists integer $r \approx \log \epsilon $ such that 
			\begin{equation}
				|	M(\xi)^{\alpha(x)}  - \sum_{q=1}^r L_q(\alpha(x)) M(\xi)^{\overline{\alpha}_q}| \leq \epsilon, 
			\end{equation}
			where  $\overline{\alpha}_q$'s are Chebyshev points in the interval $(\alpha_{\min}, \alpha_{\max})$ and $L_q(\cdot)$'s are associated Lagrange polynomials. 
		\end{lem}
		\begin{proof}
			Consider the univariant function $f(t)=a^t$ for $a>0$ and $t\in (\alpha_{\min}, \alpha_{\max})$. By the Lagrange interpolation associated with the Chebyshev points (see \cite{KLOKOV198793}), we have 
			\begin{equation}
				|	a^{t}  - \sum_{q=1}^r L_q(t) a^{\overline{\alpha}_q}|< |f^{(r)} (\theta)| [2^{r-1} r! (r-1) ]^{-1} <| (\ln  a)^r a^\theta| [2^{r-1} r! (r-1) ]^{-1} \leq \epsilon. 
			\end{equation}
			Here $\theta \in (0,1)$.
				Then replace $a$ with $M(\xi)$, and the exponent $t$ by $\alpha(x)$, respectively, which immediately leads to the desired result. 
		\end{proof}

		As direct consequence of the above lemma, we can use the superposition principle,  the linear combination  of the constant-order discrete fractional operators to approximate  the variable-order fractional Laplacian to achieve a fast computation. That is 
  \begin{equation}
(-\Delta_h)^{\alpha(x)/2} \approx \sum_{q=1}^r \mbox{diag}(L_q(\alpha(x)) ) (-\Delta_h)^{\overline{\alpha}_q/2}.
  \end{equation}

Based on the above discussion, we describe our algorithm for the numerical evaluation as follows: 

\begin{algo}
For the fractional Laplacian of the variable order, we have the following computation. 
  	\begin{itemize}
			\item Step 1,  generate the Chebyshev points $\overline{\alpha}_q \in (\alpha_{\min},\alpha_{\max})$, and evaluate Lagrange polynomials $ L_q(\alpha (x_j) )$ for $q\leq  r$ and $ x_j\in \Omega_h$ with $j_p=1,\dots,N$ for $p=1,\dots,d$. 
			
			\item Step 2, for each fixed constant-order $\overline{\alpha}_q$,  compute $(-\Delta_h)^{\overline{\alpha}_q/2} u(x_j) $ for $x_j\in \Omega_h$, and then multiply the diagonal matrix  $ L_q(\alpha(x_j)) $.
			
			\item Step 3, add them together over $q=1,2\cdots r$. 
		\end{itemize}
		
	\end{algo}	
		Note that, in Step 2,  we refer the readers to  \cite{Hao-Zhang-Du-2021} for  implementation details.

\subsection{Numerical experiments}\label{subsection-numer-exeriments-approximation}

In the following example, we take the Gaussian  function as our test function to illustrate the accuracy  of the theoretical result.  
 
		\begin{exm}[Approximations of variable-order fractional Laplacian in multi-dimensions]\label{exm-approximation-function}
			Consider the Gaussian function $u(x)=\exp(-|x|^2)$ with a truncated computational domain $[-4,4]^d$, where $d$ is for the dimension and $x=(x_1,x_2,\dots,x_d)^{\top}$. 
		\end{exm}
		
		In this experiment, the variable-order Laplacian of the Gaussian function is given as (see the Appendix for the derivation)
		\begin{eqnarray}
			\label{exact_solution}
			(-\Delta)^{\alpha(x)/2} u = (-\Delta)^{\alpha(x)/2}[\exp(-|x|^2)]=\frac{2^{\alpha(x)}\Gamma((d+\alpha(x))/2)}{\Gamma(d/2)}{_1F_1}((d+\alpha(x))/2;d/2;-|x|^2).
		\end{eqnarray}	
		 The error is measured  as $E_\infty(h) = ||(-\Delta_h)^{\alpha(\cdot)/2}u-(-\Delta)^{\alpha(\cdot)/2} u ||_{L_h^\infty}$,
		and the convergence order is estimated by $\log_2({E_\infty(2h)}/{E_\infty(h)})$.
		
		Note that the function $u(x)$ is smooth and satisfies the regularity requirement in   Theorem \ref{thm-approximation-property}. Here we examine the effect of the variable order $\alpha(x)$ on the accuracy of the finite difference approximation.  Tables \ref{Approx_VO_Lap_1D}-\ref{Approx_VO_Lap_3D} show the $\operatorname{\mathit{L_h^\infty}-\mathit{norm}}$ errors and convergence orders for different $\alpha(x)$ in 1D/2D/3D,  respectively. Three different $\alpha(x)$ with different range and smoothness are tested.  From the tables, we observe that our approximation is of second order, which agrees with the predicted convergence order in Theorem \ref{thm-approximation-property}. 
  In particular, the approximation is very robust in the sense that we are still able to get the second-order convergence even for the non-smooth piece-wise function $\alpha(x)$. 
  

		\begin{table}[h]
			\centering
			\caption{The errors and convergence orders for the approximation to  $(-\Delta)^{\alpha(x)/2}u(x)$ in 1D (Example \ref{exm-approximation-function}).  Here  $\alpha_1(x)\in (0.1,1)$ and $\alpha_2(x)\in (1,1.9)$.}
			\setlength{\tabcolsep}{11pt}{
				\begin{tabular}{lcccccc} 
					\toprule[1pt]
					\multirow{2}{*}{$h$} & \multicolumn{2}{l}{$\alpha_1(x) = 1-0.9\tanh(|x|)$}  & \multicolumn{2}{l}{$\alpha_2(x) = 1+0.9\tanh (|x|)$} &
                     \multicolumn{2}{l}{$\alpha_3(x) = 0.4\chi_{[x>0]} + 1.2\chi_{[x\leq0]}$}\\
					\cmidrule(r){2-3}  
					\cmidrule(r){4-5}
                    \cmidrule(r){6-7}
					& $E_\infty(h)$ & Order &  $E_\infty(h)$ & Order &$E_\infty(h)$ & Order\\
					\midrule[0.25pt]
					1/4 &1.17e-02 &*  &2.25e-02 &* &1.68e-02 &* \\
					1/8& 2.93e-03 &1.99   &5.69e-03 &1.98 & 4.23e-03 & 1.99\\
					1/16& 7.35e-04&2.00  & 1.44e-03 &2.00 & 1.06e-03 &2.00\\
					1/32 & 1.84e-04&2.00  &3.61e-04&2.00 &2.65e-04 &2.00\\
					1/64 & 4.61e-05&2.00 &9.03e-05&2.00 &6.62e-05 &2.00\\
     
					\bottomrule[1pt]
			\end{tabular}}
			\label{Approx_VO_Lap_1D}
		\end{table}
		
		\begin{table}[h]
			\centering
			\caption{The errors and convergence orders for the approximation to $(-\Delta)^{\alpha(x)/2}u(x)$ in 2D (Example \ref{exm-approximation-function}). Here $D = \{x\in \Omega~|~x_1,x_2>0\}$,     $\alpha_1(x)\in (0.1,1)$ and $\alpha_2(x)\in (1,1.9)$.  }
			\setlength{\tabcolsep}{11pt}{
				\begin{tabular}{lcccccc} 
					\toprule[1pt]
					\multirow{2}{*}{$h$} & \multicolumn{2}{l}{$\alpha_1(x) = 1-0.9\tanh (|x|)$}  & \multicolumn{2}{l}{$\alpha_2(x) = 1+0.9\tanh (|x|)$} & \multicolumn{2}{l}{$\alpha_3(x) = 0.4\chi_{[x\in D]}+1.2\chi_{[x\in D^c]}$}\\
					\cmidrule(r){2-3}  
					\cmidrule(r){4-5}  \cmidrule(r){6-7}
					& $E_\infty(h)$ & Order &  $E_\infty(h)$ & Order  &$E_\infty(h)$ & Order\\
					\midrule[0.25pt]
					1/4 &2.06e-02& * &2.68e-02 &  * &3.05e-02 &*\\
					1/8& 1.33e-02& 1.99  &5.19e-03 & 1.99 &7.69e-03&     1.99\\
					1/16& 3.32e-03  & 1.99 &1.31e-03 & 1.99& 1.93e-03     &     1.99\\
					1/32 & 8.28e-04 &2.00  &3.37e-04&1.97 &4.90e-04     &     1.98\\
					\bottomrule[1pt]
			\end{tabular}}
			\label{Approx_VO_Lap_2D}
		\end{table}
		
		\begin{table}[h]
			\centering
			\caption{The errors and convergence orders for the approximation to $(-\Delta)^{\alpha(x)/2}u(x)$ in 3D (Example \ref{exm-approximation-function}). Here $D = \{x\in \Omega ~|~x_1,x_2,x_3>0\}$,    $\alpha_1(x)\in (0.1,1)$ and $\alpha_2(x)\in (1,1.9)$.}
			\setlength{\tabcolsep}{11pt}{
				\begin{tabular}{lcccccc} 
					\toprule[1pt]
					\multirow{2}{*}{$h$} & \multicolumn{2}{l}{$\alpha_1(x) = 1-0.9\tanh{|x|}$}  & \multicolumn{2}{l}{$\alpha_2(x) = 1+0.9\tanh (|x|)$} &\multicolumn{2}{l}{$\alpha_3(x) = 0.4\chi_{[x\in D]}+1.2\chi_{[x\in D^c]}$}  \\
					\cmidrule(r){2-3}  
					\cmidrule(r){4-5}
					\cmidrule(r){6-7}
					& $E_\infty(h)$ & Order &  $E_\infty(h)$ & Order &  $E_\infty(h)$ & Order\\
					\midrule[0.25pt]
					$1$ & 3.98e-01     &*  &3.98e-01  & * & 5.83e-01     &* \\
					$1/2$& 1.10e-01     &1.86     & 1.43e-01    &  1.48 & 1.64e-01     &1.83\\
					$1/4$& 2.81e-02     &     1.96   & 3.97e-02     &     1.85     & 4.23e-02     &     1.96\\
					\bottomrule[1pt]
			\end{tabular}}
			\label{Approx_VO_Lap_3D}
		\end{table}

\vskip 10cm

\newpage

	\section{Application on  fractional PDEs} \label{sec-app-frac-ellip-eqn}

			\subsection{Finite difference scheme for the elliptic equations}
		
		In this section, we will apply the proposed discrete operator into solving the  elliptic equations and parabolic equations   with the variable-order fractional Laplacian. 	For the parabolic equations, after the semi-discretization in temporal direction, they are reduced to the elliptic one at the each time step. Thus we only discuss  the elliptic case as below:
		\begin{eqnarray}
			&&	(-\Delta)^{\alpha(x)/2}u+ b(x) u(x)= f(x), \quad x\in \Omega\subset \mathbb{R}^d, \qquad \alpha(x) \in (0,2), \label{eq:2d-diffu-reac} \\
			&& u(x)=0, \quad x\in \Omega^{c}, \label{eq:2d-diffu-reac-bc}	\end{eqnarray}
		where the reaction coefficient {$b(x)$ is non-negative}, 
		the right-hand-side function $f(x)$ is given,    $\Omega$ is the standard bounded domain with the sufficiently smooth boundary  and $\Omega^{c}$ is the complement of  $\Omega$ in $\mathbb{R}^d$.	 

\begin{rem}
  Regarding the mathematical theory related to the variable-order fractional Laplacian, the authors in \cite{Fukushima-Uemura-2012} showed that the kernel admits the lower bounded Dirichlet form.  Later the authors in \cite{Felsinger-KV-2015} set up the elliptic and the parabolic Dirichlet problem for the linear nonlocal operators which cover the special case as in this work. The paper formulated the problem in the classical framework of Hilbert spaces and proved unique solvability using  standard techniques like the Fredholm alternative.   In \cite{Silvestre-2005}, the author investigated the H\"older regularity of the elliptic equations with a variable-order fractional Laplace operator.  
\end{rem}

		Consider  Equation (Eq.) \eqref{eq:2d-diffu-reac} at the grid points $x_j=jh$ and we obtain
		\begin{eqnarray}
			&& 	(-\Delta)^{\alpha(x_j)/2} u(x_j)+ b(x_j) u(x_j) =f(x_j).\label{Discrete-eq-1}
		\end{eqnarray}
			Replacing the fractional Laplacian  by  the difference operator defined in Eq. \eqref{def-discrete-frac-lap}, Eq. \eqref{Discrete-eq-1} becomes
		\begin{eqnarray}
			&& 	(-\Delta_h)^{\alpha(x_j)/2} u(x_j)+b(x_j) u(x_j)  =f(x_j)+T_{u}(x_j),\quad   x_j \in \Omega_h,
			\label{Discrete-eq-2}
		\end{eqnarray}
		with the truncation error $T_u$ depending on the solution $u$. 
		By Theorem \ref{thm-approximation-property},  we assume 	$u \in  \mathcal{B}^{s+\alpha_{\max}} (\mathbb{R}^d)$ with a positive constant $s\leq 2$, such that  
		there exists a constant $c_u$   satisfying
		\begin{eqnarray}\label{trunction-err}
			|T_{u}|\leq  c_u{h^{s}}.
		\end{eqnarray}

		Omitting $T_u$ in Eq. \eqref{Discrete-eq-2} and
		denoting by $u_{j}$ the numerical approximation of $u(x_j)$,  and $\alpha(x_j)= \alpha_j$ (the short notation $\alpha_j$ will be frequently  used  throughout the following without the confusion), $b_j=b(x_j)$ and   $f_{j}=f(x_j)$, we get the finite    difference  scheme 
		\begin{eqnarray}
			&& 	(-\Delta_h)^{\alpha_j/2} u_{j}+ b_j u_{j} =f_{j}, \quad x_j \in  \Omega_h,\label{v-order-scheme}\\
			&& u_{j}=0, \quad  x_j \in  \Omega^c_h.\label{v-order-scheme-bc}
		\end{eqnarray}

		\begin{thm}[Stability and Convergence]\label{thm-stability-convergence}  
			For any $h>0$, the finite difference scheme   \eqref{v-order-scheme}-\eqref{v-order-scheme-bc} is uniquely solvable and  stable with respect to the right hand side $f$ in the following  sense 
			\begin{eqnarray}
				&&
				\|u\|_{L_h^\infty} \leq c_3 \|f\|_{L_h^\infty}. 
			\end{eqnarray}
			Moreover, let $U_{j}=u(x_j)$ be the  solution of  Eqs. \eqref{eq:2d-diffu-reac}-\eqref{eq:2d-diffu-reac-bc}  and $u_{j} $ be the numerical solution of the difference scheme \eqref{v-order-scheme}-\eqref{v-order-scheme-bc}.   Assuming  that the solution $u \in  \mathcal{B}^{s+\alpha_{\max}} (\mathbb{R}^d)$ with $s \leq 2$,
			it  holds that
			\begin{eqnarray}
				\quad  \|U-u\|_{L_h^\infty}\leq c {h^{s}}.
			\end{eqnarray}
		\end{thm}	

		\subsection{Numerical Experiments}\label{sec-num-experm}
		
		In this subsection,   several  numerical examples  are presented to show the efficiency  and accuracy  of the schemes. 	 We first apply the proposed finite difference scheme to solve the elliptic equation in Example \ref{Ex-known-solution} and  the time-dependent problem in Example \ref{Parabolic_Equation-2D}, respectively. To show the wide application of finite difference method, the finite difference  scheme is applied to solve Allen-Cahn equation in Example \ref{AppliedtoFACequation}. In addition, we consider the coexistence of anomalous diffusion problem in Example \ref{Ex-general-domain}. In the last Example \ref{parabolic-eq-3-D}, we turn to the time-dependent problem in  3D.
		
		Throughout the examples, the low-rank approximation terms $r$ is taken as $7$ to compute the Chebyshev points and the quadrature number $M$ is taken as $2^{14}$ to compute the coefficients $a^{(\alpha_j)}_{n}$ in Step 2  so that the accuracy of numerical solution will not be polluted. The tolerance of the BiCGSTAB (see \cite{VdV92}) method is set as $10^{-16}$ and the initial guess is fixed as zero in our simulations. All the simulations are performed in MATLAB 2021b on a 64-bit Ubuntu with Intel Xeon(R) Gold 6240 CPU processor at 2.60 GHz and 126 GB of RAM.

		\begin{exm}[Finite difference scheme for variable-order fractional elliptic equations]\label{Ex-known-solution}
			Consider the elliptical equations with variable-order fractional Laplacian $(-\Delta)^{\alpha(x)/2}  u+\mu u=f$ on $\Omega = [-1,1]^2$. Case 1: $\mu=1$ and the exact solution is set as $u(x)= (1-x_1^2)^\beta (1-x_2^2)^\beta$; Case 2:  $\mu=0$, the exact solution is unknown and the right-hand data is $f(x)=1$.
		\end{exm}	 
		
		In Case 1, we take $\beta = 4$. Since the right hand side $f(x)$ is unknown, we need to compute it with very fine step-size which is set as $f^{ref}\approx f_h=(-\Delta_h)^{\alpha(\cdot)/2}  u+u $ with small $h=2^{-9}$. In Table \ref{VOFDE_Lap_solve_case1}, the $\operatorname{\mathit{L_h^\infty}-\mathit{norm}}$ errors and convergence orders are listed with different $\alpha(x)$. Here,  the $\operatorname{\mathit{L_h^\infty}-\mathit{norm}}$ error is measured as $E_\infty(h) = ||u_h-u||_{L_h^\infty}$. Furthermore, from Table \ref{VOFDE_Lap_solve_case1}, we can observe that the convergence order of our scheme is of second order.
		
		Next we consider Case 2, where $f(x)=1$ and the exact solution $u(x)$ is unknown. Table \ref{VOFDE_Lap_solve_case2} shows the errors and convergence orders with different $\alpha(x)$. Here, the $\operatorname{\mathit{L_h^\infty}-\mathit{norm}}$ error is measured as $E_\infty(h) = ||u_h-u_{h/2}||_{L_h^\infty}$. From the table,   we observe that the convergence order is less than second order and is around $\tilde{\alpha}$, where $ \min\limits_{x\in\Omega} \alpha(x)/2 \leq \tilde{\alpha} \leq \max\limits_{x\in\Omega} \alpha(x)/2$. 

        In particular, if $\alpha(x)$ is a non-smooth piece-wise function, Table \ref{VOFDE_Lap_solve_case1}  shows that our scheme \eqref{v-order-scheme}-\eqref{v-order-scheme-bc}  still maintains second-order convergence when the exact solution $u(x)$ is smooth enough.  But  the convergence order reduces to $\min\limits_{x\in\Omega} \alpha(x)/2$ when the right-side data $f(x)=1$ and the exact solution $u(x)$ is unknown; see Table \ref{VOFDE_Lap_solve_case2}. The reason for the reduced convergence is that the solution may have the interior singularity inherited from the variable-order function.   
        
       To recover the second-order convergence rate, we take $\alpha(x)\in C(\mathbb{R}^2)$ satisfying $\alpha(x) = 2$ on $\partial \Omega$ and $\alpha(x)<2$ in $\Omega$.  Figure \ref{VO_Ex3_alpha} shows the profile of the function $\alpha(x)=0.8+1.2\max{(\abs{x_1},\abs{x_2})}$. From  Table \ref{VOFDE_Lap_solve_case2_special_order}, we observe that the appropriately selected variable-order function $\alpha(x)$  enables us to recover the second-order convergence. We remark that for the case 1 corresponding to the first column results  in Table \ref{VOFDE_Lap_solve_case2_special_order}, we further test the smaller case $h=1/128$ and $1/256$ and the convergence orders are 1.69 and 3.02, respectively. Thus the average order is around 2.  For the case 2 in the second column, we also take smaller stepsize $h=1/128$ and the convergence order is 2.1.      This implies that setting  the variable-order $\alpha(x)$ can change the regularity of the solution and avoid the non-physical boundary singularity which reduces the convergence order as in Table \ref{VOFDE_Lap_solve_case2}.

		\begin{table}[h]
			\centering
			\caption{The errors and convergence orders of the scheme \eqref{v-order-scheme}-\eqref{v-order-scheme-bc} for $(-\Delta)^{\alpha(x)/2}u(x) + u(x)=f(x)$ by the BiCGSTAB method. (Case 1 of Example \ref{Ex-known-solution}).}
			\setlength{\tabcolsep}{13pt}{
				\begin{tabular}{lcccccc} 
					\toprule[1pt]
					\multirow{2}{*}{$h$} & \multicolumn{2}{l}{$\alpha(x) = 1+|x|/4$}  & \multicolumn{2}{l}{$\alpha(x) = 1-0.5\tanh (|x|)$} & \multicolumn{2}{l}{$\alpha(x) = 0.4\chi_{[x_1\leq0]}+1.2\chi_{[x_1>0]}$}\\
					\cmidrule(r){2-3}  
					\cmidrule(r){4-5}
                    \cmidrule(r){6-7} 
					& $E_\infty(h)$ & Order &  $E_\infty(h)$ & Order&  $E_\infty(h)$ & Order\\
					\midrule[0.25pt]
					1/4 &2.26e-02&*&1.86e-02&*&1.15e-02 &* \\
					1/8&5.61e-03&2.01  & 4.46e-03&2.06&4.01e-03&1.52   \\
					1/16&1.40e-03&2.00  &1.11e-03&2.01&1.04e-03&1.94  \\
					1/32 &3.51e-04&2.00 &2.78e-04&1.99&2.63e-04&1.99\\
					\bottomrule[1pt]
			\end{tabular}}
			\label{VOFDE_Lap_solve_case1}
		\end{table}
		
		
		\begin{table}[h]
			\centering
			\caption{The errors and convergence orders of the finite difference scheme for $(-\Delta)^{\alpha(x)/2}u(x)=1$ by the BiCGSTAB method. Here $D = \{x\in \Omega ~|~-0.8\leq x_1\leq 0.8,\,-0.8\leq x_2 \leq 0.8\}$(Case 2 of Example \ref{Ex-known-solution}).}
			\setlength{\tabcolsep}{13pt}{
				\begin{tabular}{lcccccc} 
					\toprule[1pt]
					\multirow{2}{*}{$h$} & \multicolumn{2}{l}{$\alpha(x) = 1+|x|/2$}  & \multicolumn{2}{l}{$\alpha(x) = 1-0.5\tanh (|x|)$} & \multicolumn{2}{l}{$\alpha(x) = 1.6\chi_{x\in D}+2\chi_{x\in D^c}$}\\
					\cmidrule(r){2-3}  
					\cmidrule(r){4-5}
					\cmidrule(r){6-7}
					& $E_\infty(h)$ & Order &  $E_\infty(h)$ & Order & $E_\infty(h)$ & Order\\
					\midrule[0.25pt]
					1/8 &6.88e-03&*  &3.40e-02&*  &1.28e-02 &*  \\
					1/16& 4.33e-03&0.67   &2.53e-02&0.43  &6.20e-03     &     1.05  \\
					1/32&2.68e-03&0.69 &1.95e-02&0.38 &3.68e-03     &     0.75 \\
					1/64 &1.63e-03&0.71 &1.56e-02&0.32 &1.96e-03     &     0.91  \\
					\bottomrule[1pt]
			\end{tabular}}
			\label{VOFDE_Lap_solve_case2}
		\end{table}

        \begin{table}[h]
			\centering
			\caption{The errors and convergence orders of the finite difference scheme for $(-\Delta)^{\alpha(x)/2}u(x)=1$ by the BiCGSTAB method. Here $g(x) = \max\{|x_1|,|x_2|\}$. (Case 2 of Example \ref{Ex-known-solution}). }
			\setlength{\tabcolsep}{8pt}{
				\begin{tabular}{lcccccccc} 
					\toprule[1pt]
					\multirow{2}{*}{$h$} & \multicolumn{2}{l}{$\alpha(x) = 0.8 + 1.2g(x)$}  & \multicolumn{2}{l}{$\alpha(x) = 1.2 + 0.8g(x)$} & \multicolumn{2}{l}{$\alpha(x) = 1.6 + 0.4g(x)$}&
                    \multicolumn{2}{l}{$\alpha(x) = 2$}\\
					\cmidrule(r){2-3}  
					\cmidrule(r){4-5}
					\cmidrule(r){6-7}
                    \cmidrule(r){8-9}
					& $E_\infty(h)$ & Order &  $E_\infty(h)$ & Order & $E_\infty(h)$ & Order & $E_\infty(h)$ & Order\\
					\midrule[0.25pt]
					1/8 &7.38e-03&*  &2.25e-02&*  &1.19e-03&* &2.65e-03&* \\
					1/16& 2.74e-03&1.43  &7.48e-03&1.59  &2.90e-04& 2.03 &6.76e-04 &1.97 \\
					1/32&9.36e-04&1.55 &2.26e-03&1.73 &7.14e-05& 2.02 &1.70e-04 &1.99 \\
					1/64 &2.99e-04&1.65 &6.46e-04&1.81 &1.76e-05&  2.02 &4.25e-05 &2.00\\
					\bottomrule[1pt]
			\end{tabular}}
			\label{VOFDE_Lap_solve_case2_special_order}
		\end{table}

        \begin{figure}[h]
			\centering\includegraphics[width=0.5\linewidth]{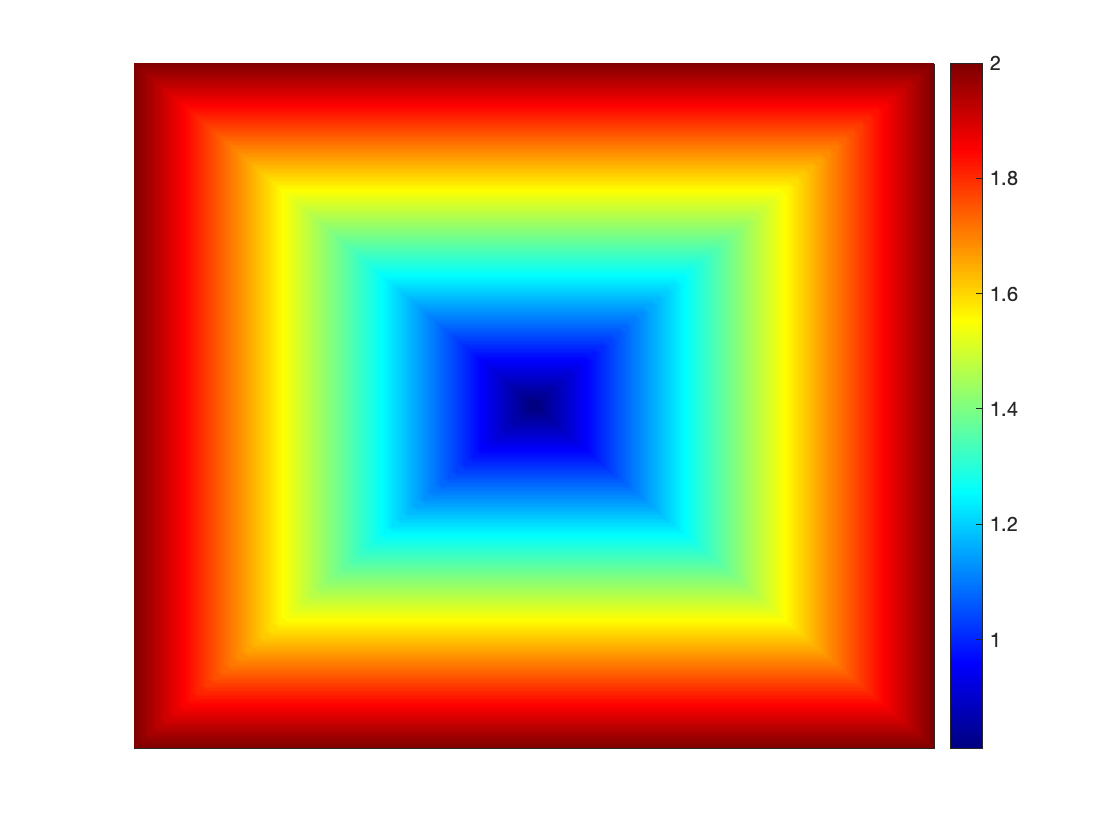}
				\caption[Figure3]{
					The profile of  the function $\alpha(x)=0.8+1.2g(x)$ with $g(x) = \max\{|x_1|,|x_2|\}$. (Example \ref{Ex-known-solution}).}
				\label{VO_Ex3_alpha}
		\end{figure}
  
		\begin{exm}[Time-dependent problem]\label{Parabolic_Equation-2D}
			Consider the initial-boundary value problem with variable-order fractional Laplacian
			\begin{eqnarray}
				&&	u_t+(-\Delta)^{\alpha(x)/2}{u}= 0, \quad  (x,t)\in \mathbb{R}^2\times (0,T], \label{eq-2d-IBVP} \\
				&& \lim_{x\rightarrow \infty} u(x,t) = 0 ,\quad   t \in (0,T) \label{eq-2d-IBVP-ic}, \end{eqnarray}
			with the initial condition $ {u}(x,0) = e^{-|x|^2},\quad x\in \mathbb{R}^2$.
		\end{exm}
		In this example, we test the convergence order of spatial direction at the final time $T=N \Delta t =0.5$, where $\Delta t$ is time step-size.
		Since the solution  decays to zero, we truncate the finite domain  in space,    $\Omega=[-4,4]^2$   as our computational domain. For time discretization we use Crank-Nicolson scheme. Since the exact solution is unknown, the error is measured as $ \|u^N(\Delta t, h)-u^N({\Delta t/2,h/2}) \|_{L_h^\infty}$.
		
		Table \ref{VO_Lap_Time_dependent_2D} shows the errors and convergence orders in $\operatorname{\mathit{L_h^\infty}-\mathit{norm}}$ for different $\alpha(x)$,  from which the second-order convergence can be observed. This is in agreement with the second-order approximation of our theoretic prediction.  
		
		\begin{table}[h]
			\centering
			\caption{The spatial errors and convergence orders for initial-boundary value problem \eqref{eq-2d-IBVP}-\eqref{eq-2d-IBVP-ic} at final time T = 0.5. (Example \ref{Parabolic_Equation-2D}). }
			\setlength{\tabcolsep}{25pt}{
				\begin{tabular}{llllll} 
					\toprule[1pt]
					\multirow{2}{*}{$h$} &\multirow{2}{*}{$\Delta t$}  & \multicolumn{2}{l}{$\alpha(x) = 1+|x|/10$}  & \multicolumn{2}{l}{$\alpha(x) = 1-0.5\tanh (|x|)$} \\
					\cmidrule(r){3-4}  
					\cmidrule(r){5-6}
					& &error & Order & error & Order\\
					\midrule[0.25pt]
					1/2 & 1/2 &1.34e-02&*  &2.36e-02&*   \\
					1/4 & 1/4&3.07e-03&2.12   &4.54e-03&2.38   \\
					1/8& 1/8&7.85e-04&1.97 &1.12e-03&2.02  \\
					1/16 & 1/16&1.99e-04&1.98  &2.82e-04&1.99 \\
					\bottomrule[1pt]
			\end{tabular}}
			\label{VO_Lap_Time_dependent_2D}
		\end{table}
		
		\begin{exm}[Allen-Cahn equation with the
  variable-order fractional  Laplacian]\label{AppliedtoFACequation}
			Consider the following nonlocal Allen-Cahn equation which is used in modeling phase field problems arising in materials science and fluid dynamics
			\begin{eqnarray}
				&&u_t +(-\Delta)^{\alpha(x)/2}u =- \frac{1}{\kappa^2} (u^3 - u), \quad  (x,t)\in \Omega\times (0,T], \label{AC1}\\
				&& u=-1, \quad  (x,t)\in   \Omega^c \times (0,T]. \label{AC2}
			\end{eqnarray}
		\end{exm}
		The computational domain is  $\Omega=[0,1]^2$, and $u$ is the phase field function. The initial condition is chosen as
		\begin{equation}
			u(x,0) = 1- \mathrm{tanh}( {d_1(x)}/{2 \kappa} ) - \mathrm{ tanh} ( {d_2(x)}/{2 \kappa} ),
			\nonumber
		\end{equation}
		where the constant $\kappa$ describes the diffuse interface width, the function $d_i(x) = \sqrt{(x_1- a_{i})^2 +(x_2-b_{i})^2 }$ and $a_{1}=b_{1}=0.42, a_{2}=b_{2}=0.58 $. Here, we apply our method to study the benchmark problem coalescence of two "kissing" bubbles in the phase field models. Letting $\overline{u} = u + 1$, we can reformulate the problem (\ref{AC1}) as 
		an equation of $\overline{u}$ with the extended homogeneous boundary conditions. Three-level linearized finite-difference scheme (see \cite{Wang-Hao-Du-2022})  is used  in the simulation.
		
		In the simulations, we take the mesh size $h=2^{-9}$ and the time step $\Delta t = 10^{-4}$. Figure \ref{VO_AC_Lap} shows the dynamics of the two bubbles described by the variable-order fractional Allen-Cahn equations with $\kappa = 0.01$.  Initially, the two bubbles are centered at points $(0.42,0.42)$ and $ (0.58,0.58)$, respectively. When $\alpha(x) = 1.8+|x|/8 \in [1.8,1.925]$, the two bubbles first coalesce into one bubble, and then the newly formed bubble shrinks and finally is absorbed by the fluid (see the  middle column in Figure \ref{VO_AC_Lap}). However, the smaller the value of the variable-order $\alpha(x)$ becomes, the slower coalescence and disappearing of the two bubbles get (see the right column in Figure \ref{VO_AC_Lap}). By contrast, if $\alpha(x) = 1.5-0.2\tanh(|x|) \in [1.3,1.5]$, the two bubbles do not merge, and they eventually vanish at the different time due to the inhomogeneous distribution of the variable-order function $\alpha(x)$ on $\bar{\Omega}$ (see the left column in Figure \ref{VO_AC_Lap}). From Figure \ref{VO_AC_Lap_evolution}, we can see the two single bubbles will be absorbed by the fluid. Figure \ref{VO_AC_Lap_profile_of_alpha} shows the profiles of different $\alpha(x)  $ on $ \bar{\Omega}$ tested in the simulation.
		
		\begin{figure}[h]
			\centering
			\subfigure[$t=0$]{
				\begin{minipage}[t]{0.3\linewidth}
					\centering
					\includegraphics[width=1\linewidth]{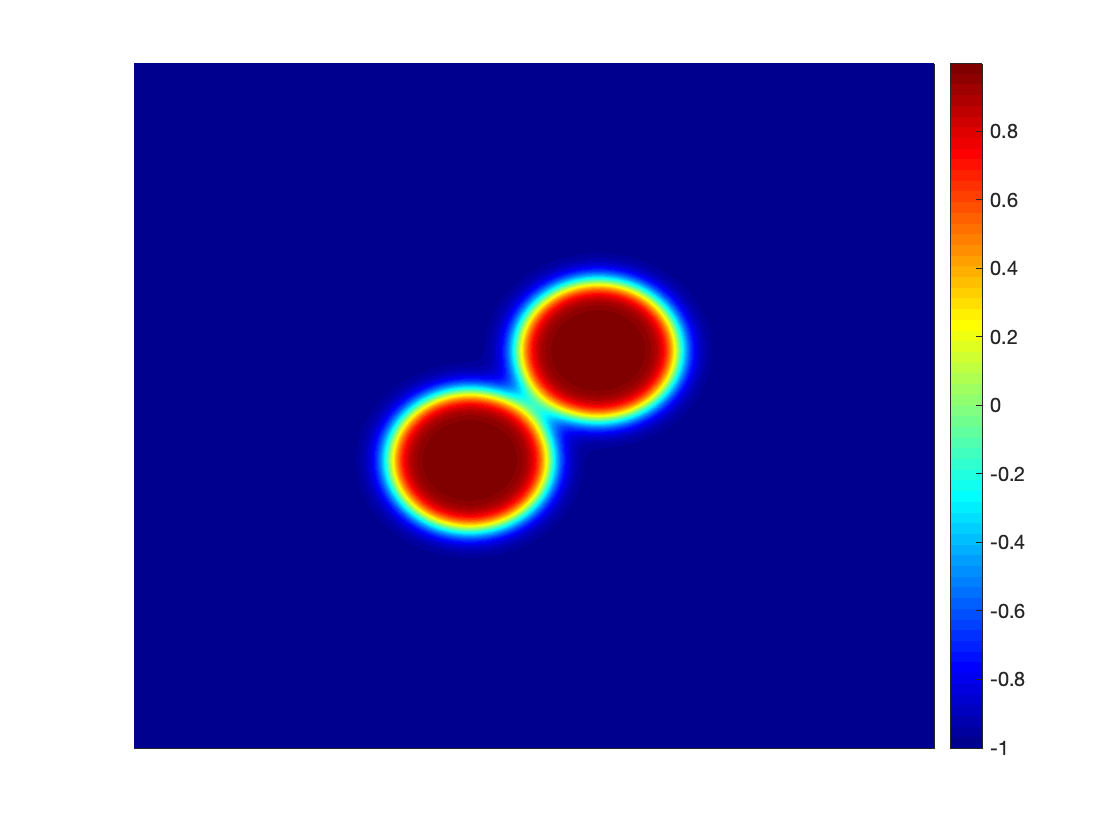}
				\end{minipage}
			}
			\subfigure[$t=0$]{
				\begin{minipage}[t]{0.3\linewidth}
					\centering
					\includegraphics[width=1\linewidth]{FDM_v_order_picture/Ex4_initial.png}
				\end{minipage}
			}
			\subfigure[$t=0$]{
				\begin{minipage}[t]{0.3\linewidth}
					\centering
					\includegraphics[width=1\linewidth]{FDM_v_order_picture/Ex4_initial.png}
				\end{minipage}
			}
			\subfigure[$t=0.001$]{
				\begin{minipage}[t]{0.3\linewidth}
					\centering
					\includegraphics[width=1\linewidth]{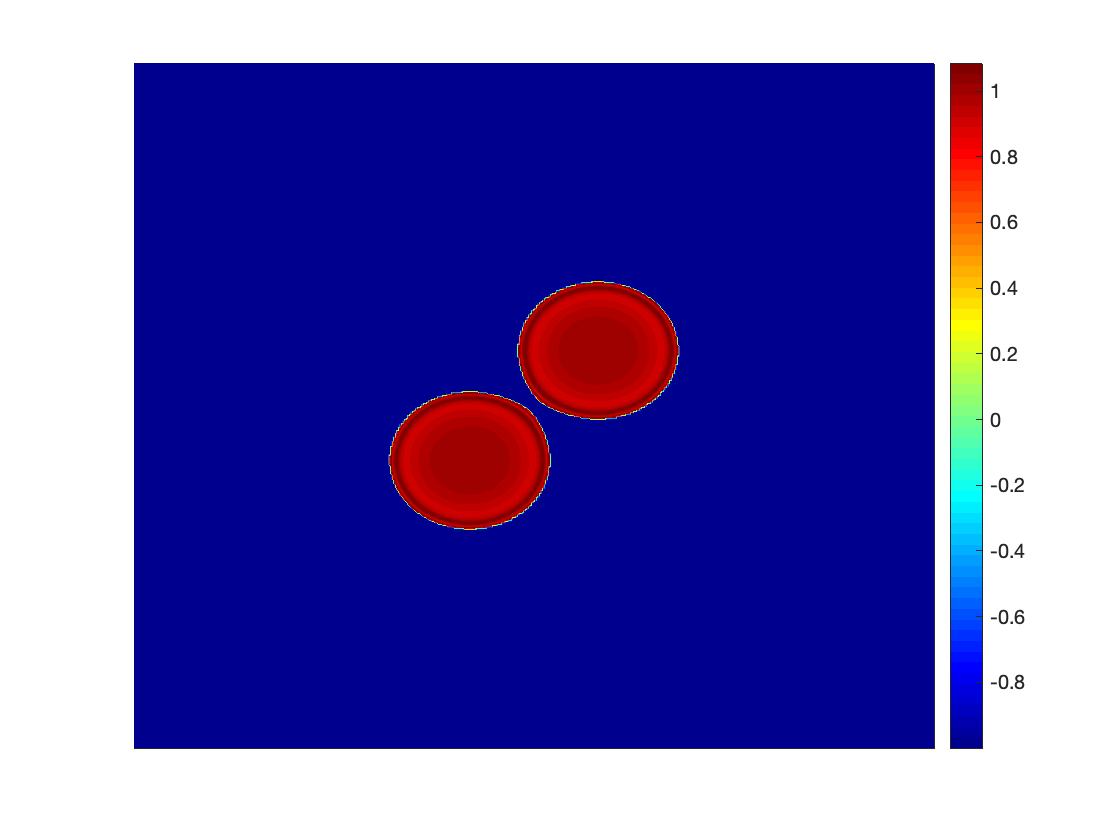}
				\end{minipage}
			}
			\subfigure[$t=0.001$]{
				\begin{minipage}[t]{0.3\linewidth}
					\centering
					\includegraphics[width=1\linewidth]{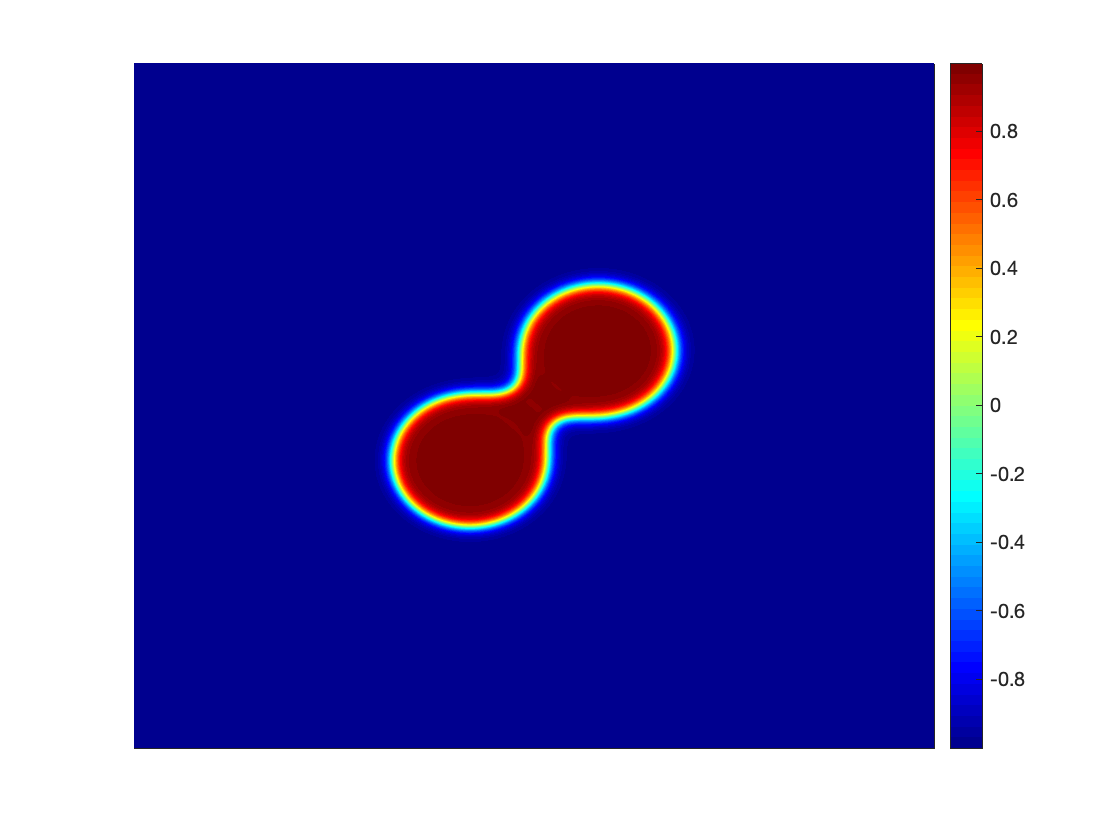}
				\end{minipage}
			}
			\subfigure[$t=0.001$]{
				\begin{minipage}[t]{0.3\linewidth}
					\centering
					\includegraphics[width=1\linewidth]{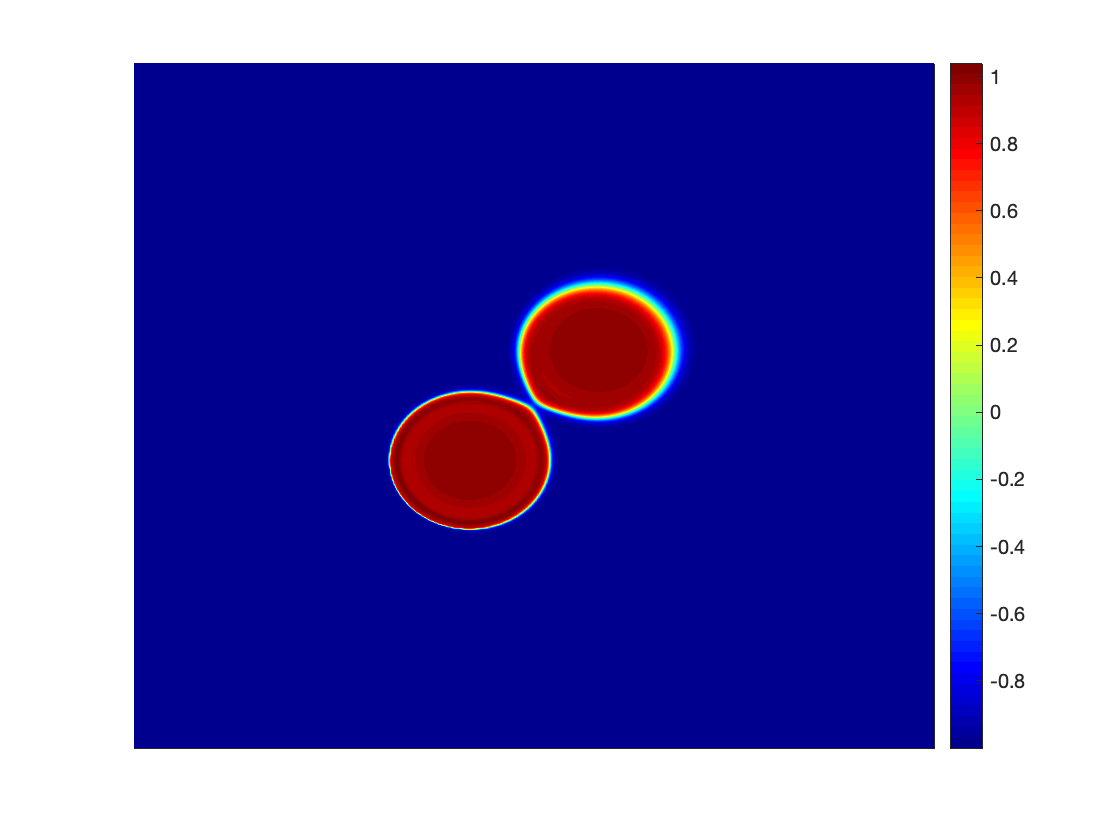}
				\end{minipage}
			}
			\subfigure[$t=0.004$]{
				\begin{minipage}[t]{0.3\linewidth}
					\centering
					\includegraphics[width=1\linewidth]{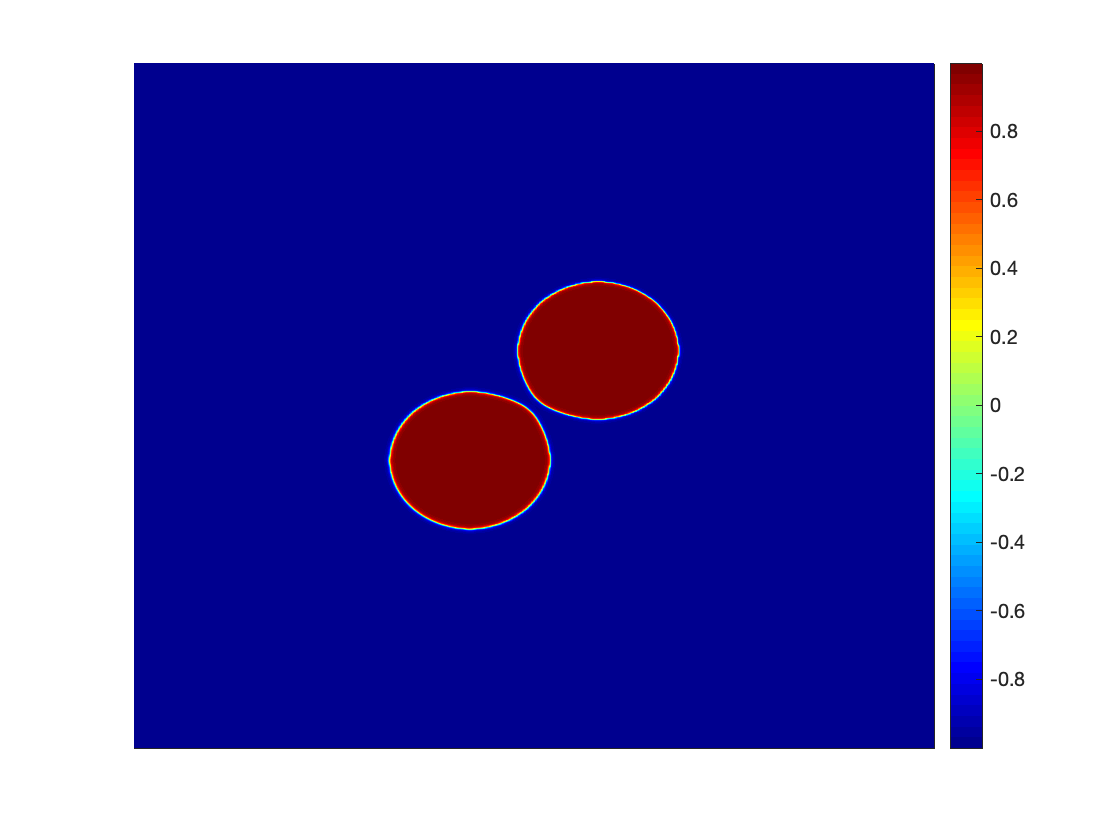}
				\end{minipage}
			}
			\subfigure[$t=0.004$]{
				\begin{minipage}[t]{0.3\linewidth}
					\centering
					\includegraphics[width=1\linewidth]{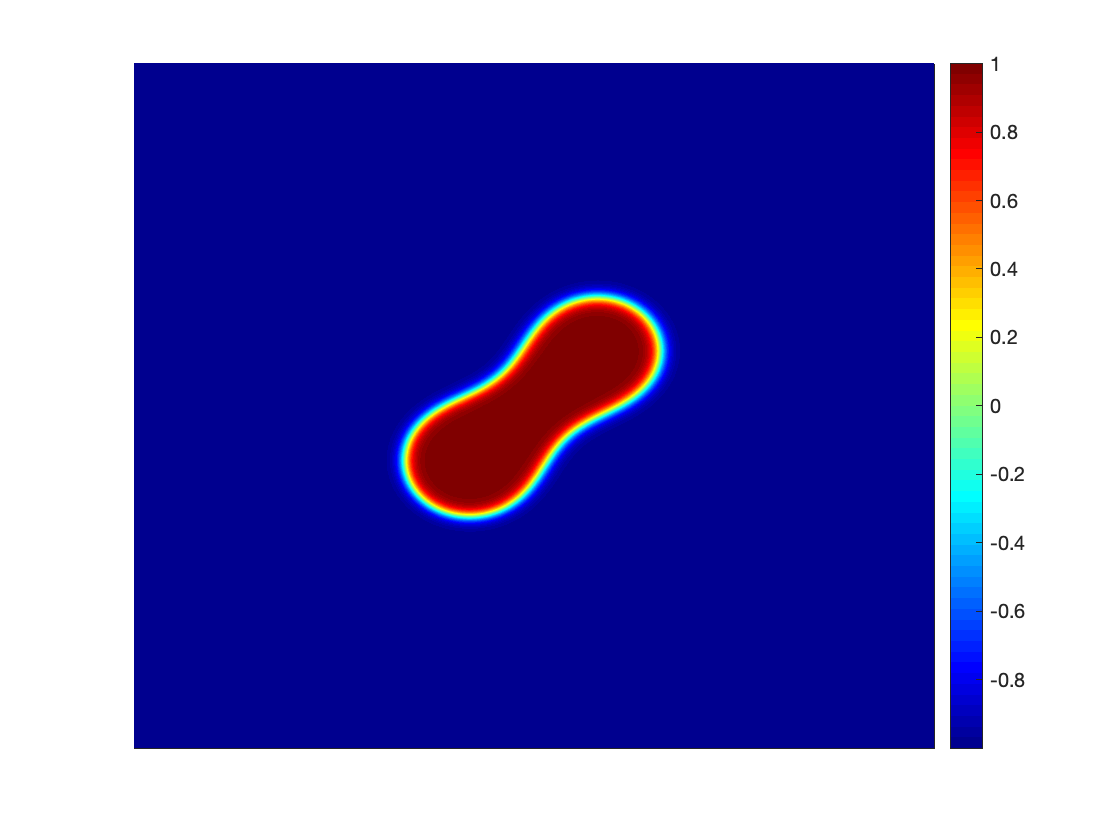}
				\end{minipage}
			}
			\subfigure[$t=0.004$]{
				\begin{minipage}[t]{0.3\linewidth}
					\centering
					\includegraphics[width=1\linewidth]{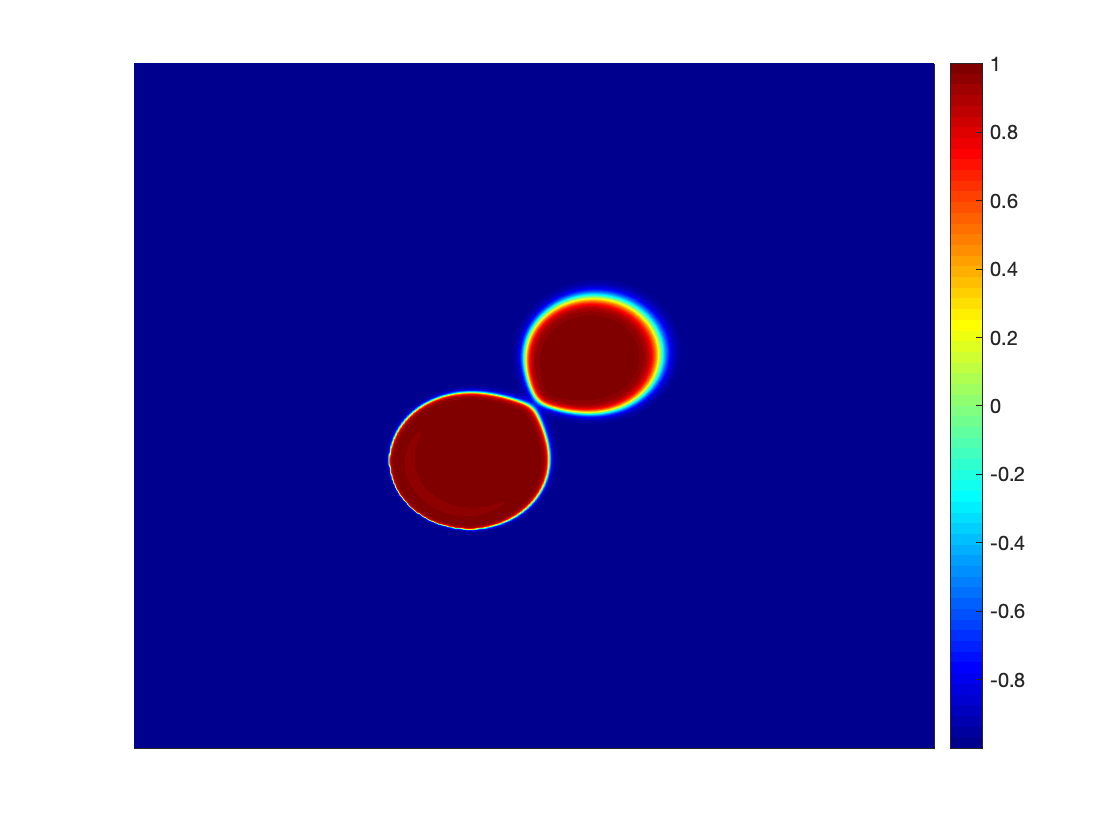}
				\end{minipage}
			}
			\subfigure[$t=0.01$]{
				\begin{minipage}[t]{0.3\linewidth}
					\centering
					\includegraphics[width=1\linewidth]{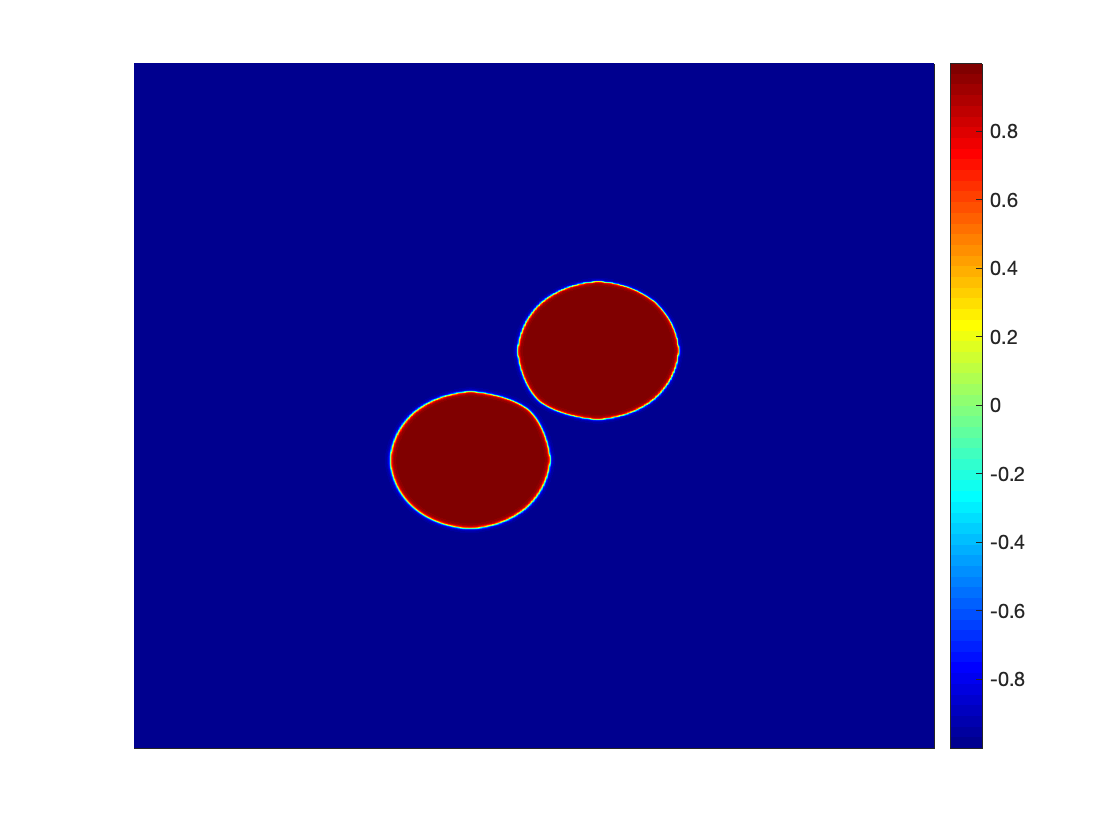}
				\end{minipage}
			}
			\subfigure[$t=0.01$]{
				\begin{minipage}[t]{0.3\linewidth}
					\centering
					\includegraphics[width=1\linewidth]{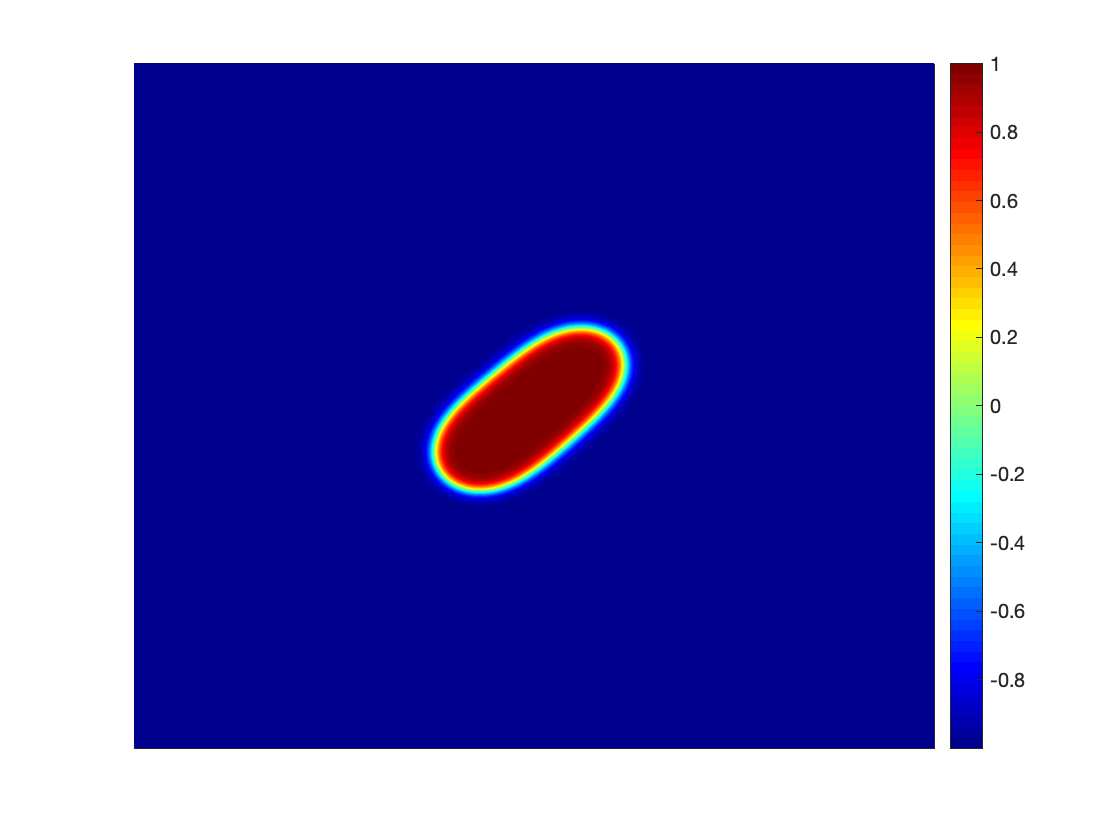}
				\end{minipage}
			}
			\subfigure[$t=0.01$]{
				\begin{minipage}[t]{0.3\linewidth}
					\centering
					\includegraphics[width=1\linewidth]{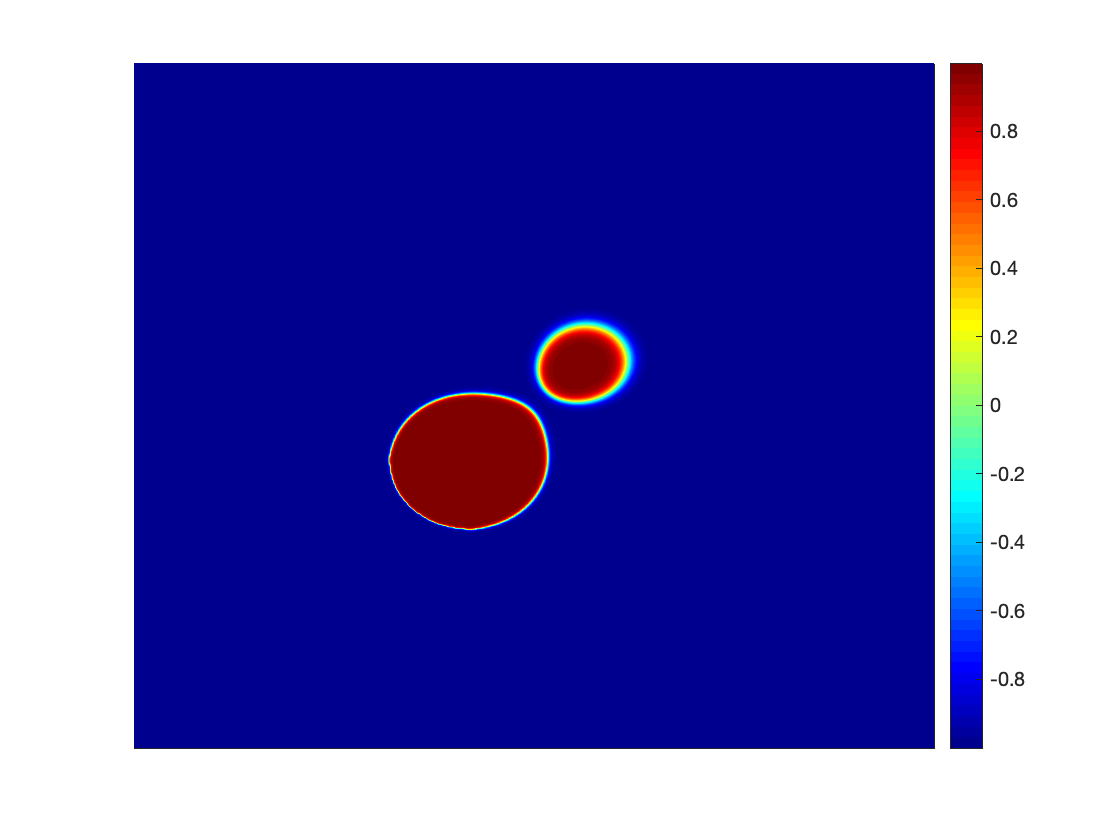}
				\end{minipage}
			}
			\caption{Dynamics of the two kissing bubbles for variable-order fractional Allen-Cahn equation (left: $\alpha(x)=1.5-0.2\tanh(|x|); $ middle: $\alpha(x)=1.8+|x|/8;$ right: $\alpha(x)=0.2\tanh(10(x_1-0.5)) + 0.2\tanh(10(x_2-0.5)) + 1.5$). (Example \ref{AppliedtoFACequation}).}
			\label{VO_AC_Lap}
		\end{figure}
		
		\begin{figure}[h]
			\centering
			\subfigure[$t=0.05$]{
				\begin{minipage}[t]{0.3\linewidth}
					\centering
					\includegraphics[width=1\linewidth]{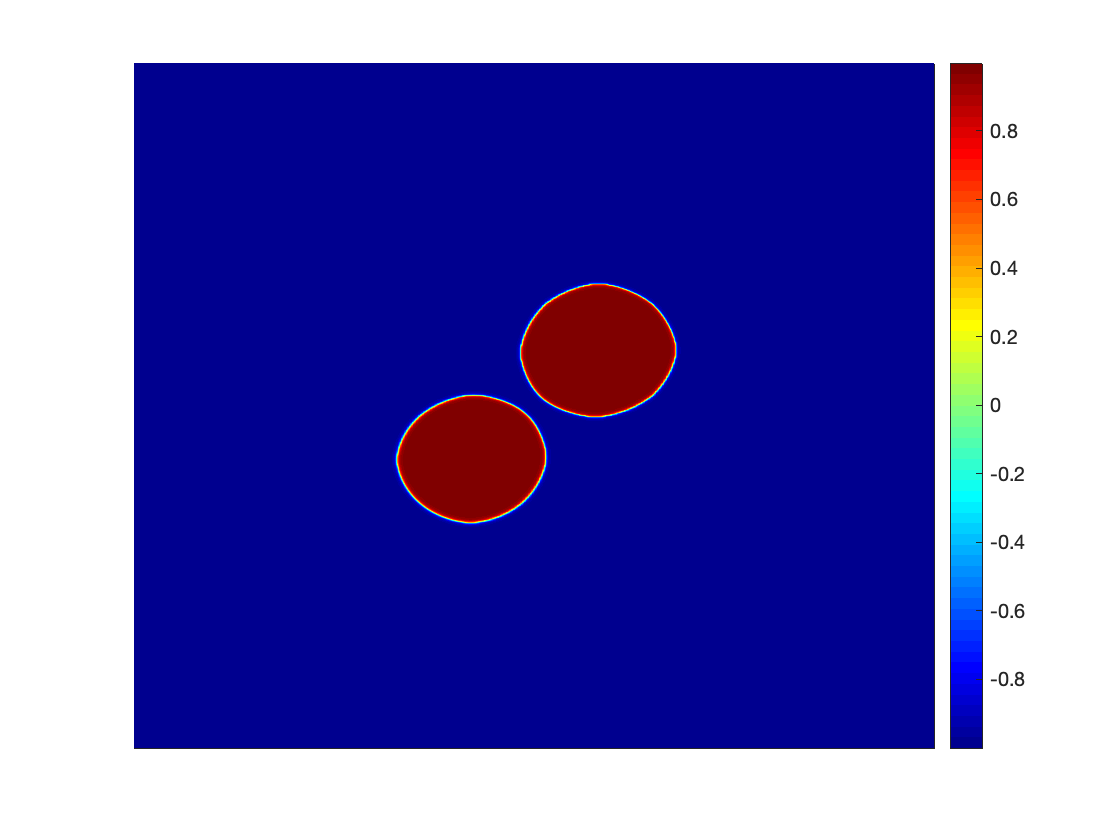}
				\end{minipage}
			}
			\subfigure[$t=0.1$]{
				\begin{minipage}[t]{0.3\linewidth}
					\centering
					\includegraphics[width=1\linewidth]{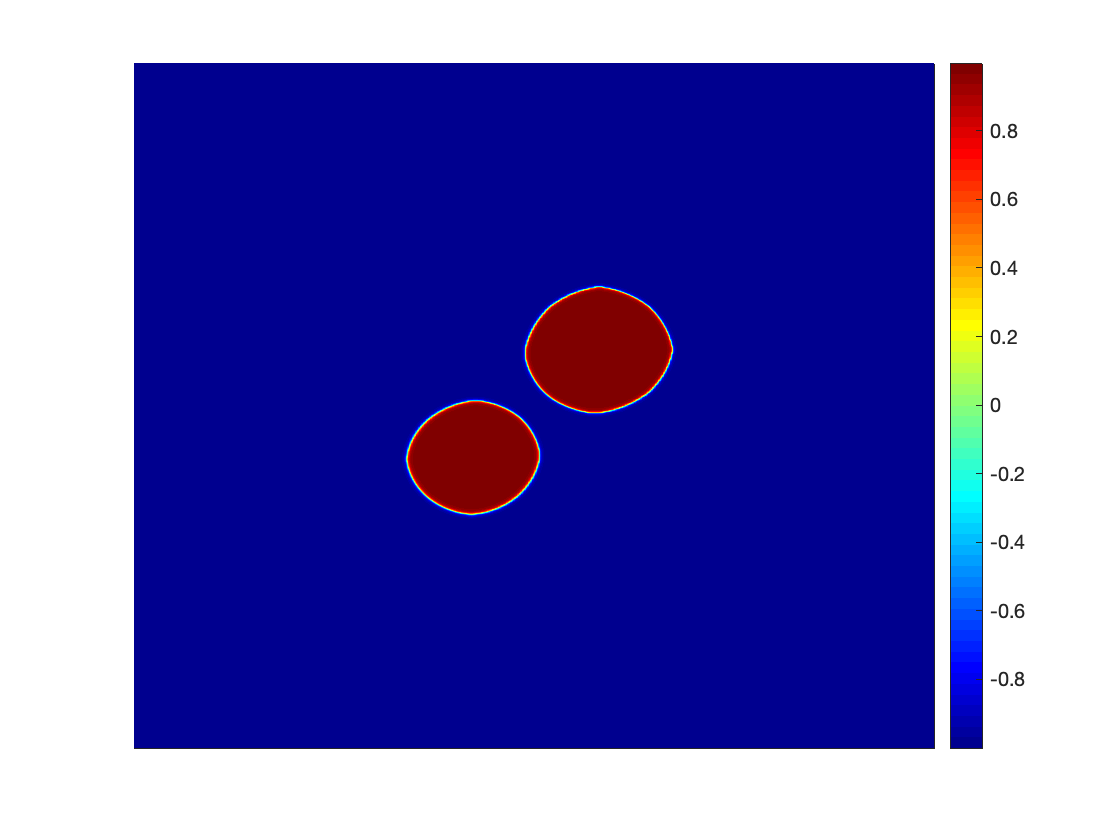}
				\end{minipage}
			}
			\subfigure[$t=0.2$]{
				\begin{minipage}[t]{0.3\linewidth}
					\centering
					\includegraphics[width=1\linewidth]{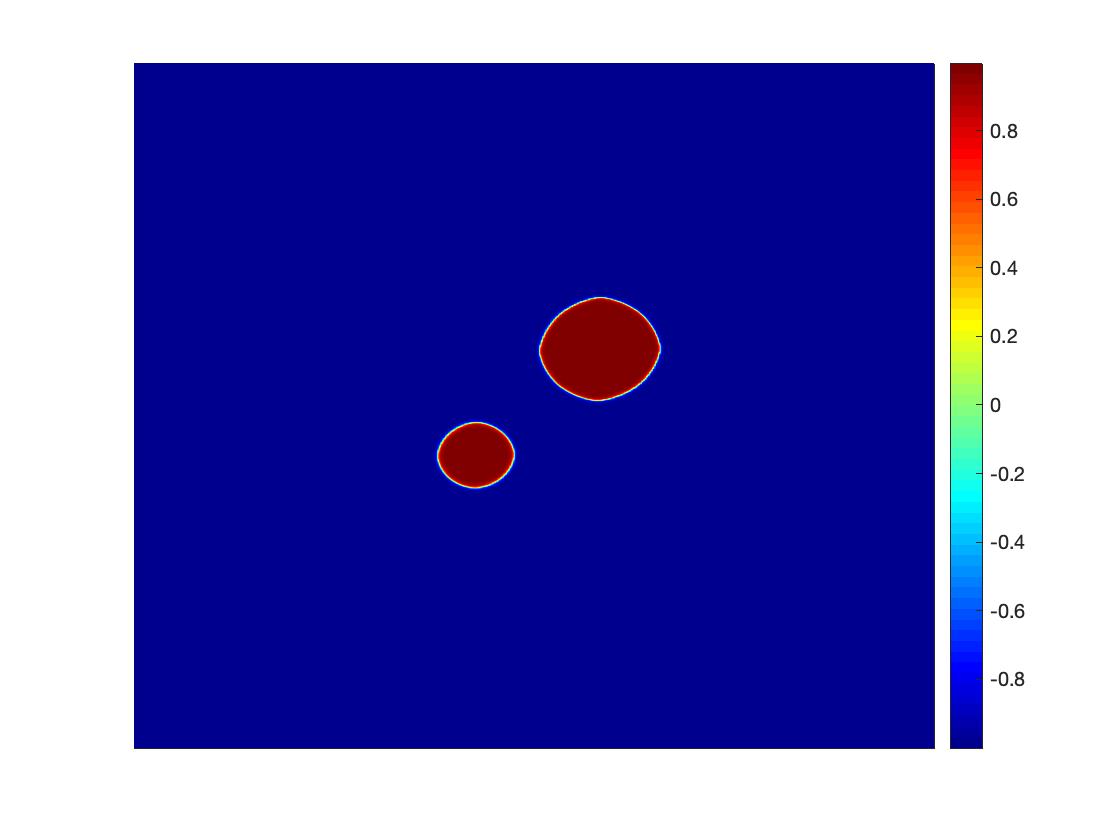}
				\end{minipage}
			}
			\caption{Evolution of the two “kissing” bubbles for variable fractional Allen-Cahn equation with $\alpha(x)=1.5-0.2\tanh(|x|)$. (Example \ref{AppliedtoFACequation}).}
			\label{VO_AC_Lap_evolution}
		\end{figure}
		
		\begin{figure}[h]
			\centering
			\subfigure[]{
				\begin{minipage}[t]{0.3\linewidth}
					\centering
					\includegraphics[width=1\linewidth]{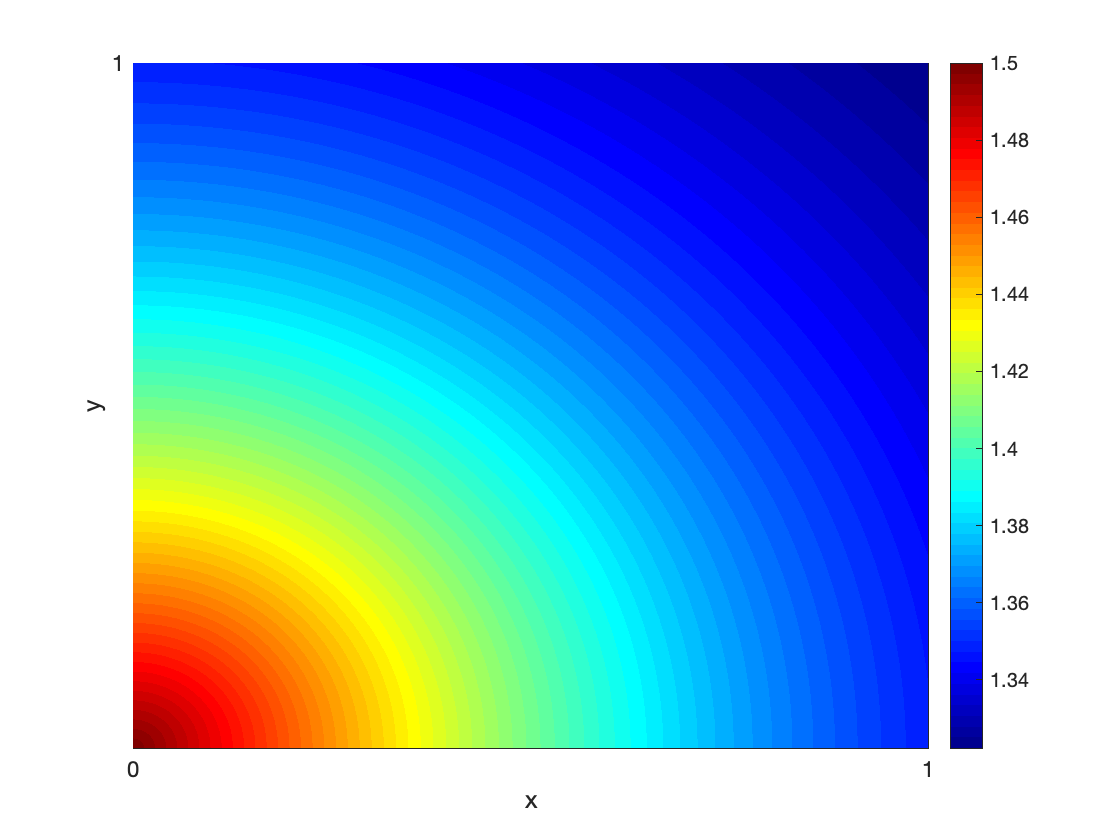}
				\end{minipage}
			}
			\subfigure[]{
				\begin{minipage}[t]{0.3\linewidth}
					\centering
					\includegraphics[width=1\linewidth]{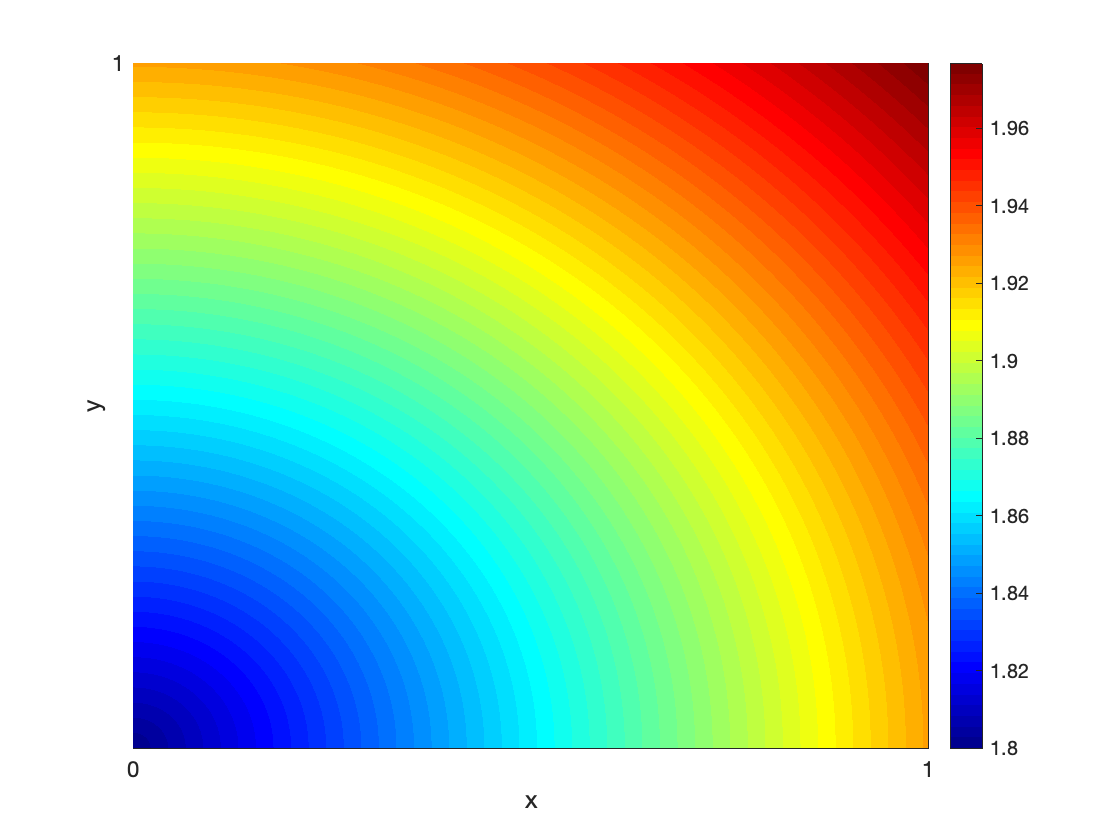}
				\end{minipage}
			}
			\subfigure[]{
				\begin{minipage}[t]{0.3\linewidth}
					\centering
					\includegraphics[width=1\linewidth]{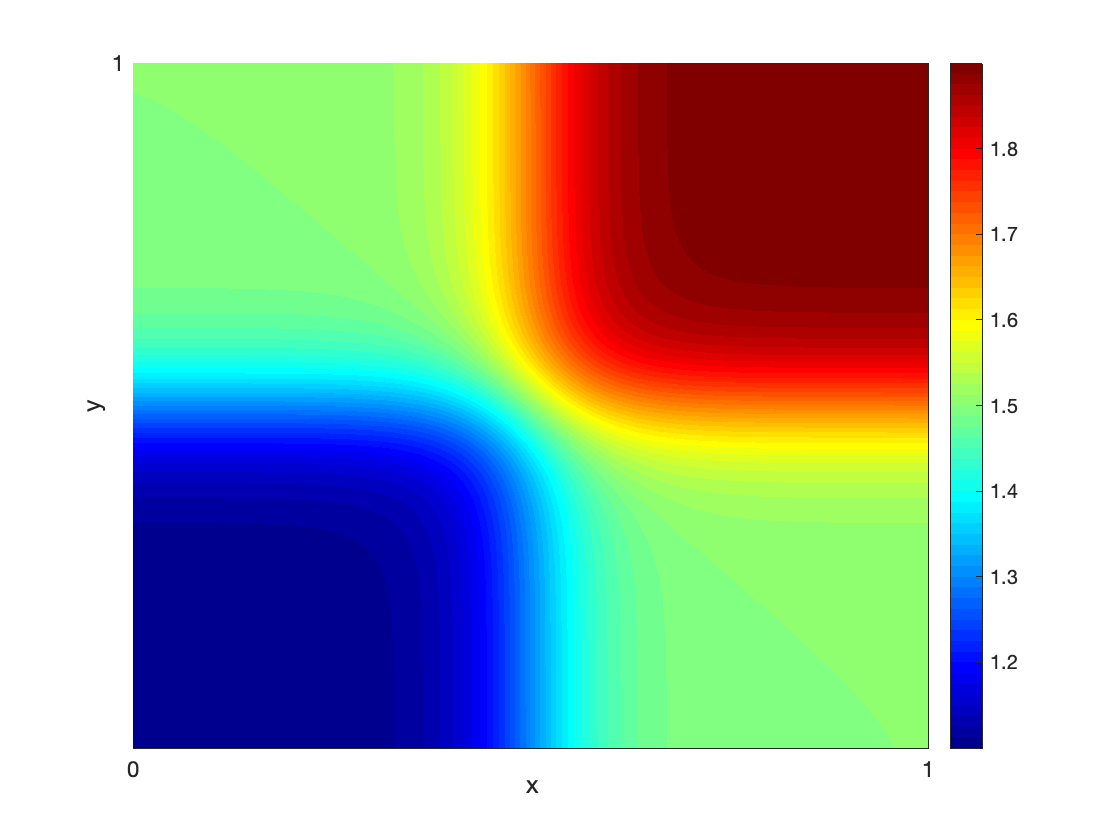}
				\end{minipage}
			}
			\caption{Profiles of order $\alpha(x) \, {\rm on} \, \Omega = [-1,1]^2$. (left: $\alpha(x)=1.5-0.2\tanh(|x|); $ middle: $\alpha(x)=1.8+|x|/8;$ right: $\alpha(x)=0.2\tanh(10(x_1-0.5)) + 0.2\tanh(10(x_2-0.5)) + 1.5$). (Example \ref{AppliedtoFACequation}).}
			\label{VO_AC_Lap_profile_of_alpha}
		\end{figure}
		
		\begin{exm}[Coexistence of anomalous diffusion problem \cite{lenzi2016anomalous}]\label{Ex-general-domain}
			
			Consider the coexistence of anomalous diffusion problem:
			\begin{eqnarray}
				&&\partial_t u(x, t)=-\kappa(-\Delta)^{ \alpha(x)/2 } u,  \quad \, x \in \Omega,\, t>0, \nonumber\\
				&&u(x, t)=0,  \quad \, x \in \mathbb{R}^2 \backslash \Omega,\, t \geq 0.\nonumber 
			\end{eqnarray}
		\end{exm}
		
		In the simulation,  we extend the irregular domain $\Omega$ (Figure \ref{VO_Lap_Coexistence_initial}) into a large rectangular one  such that $\Omega \in \mathbb{R}^2$. Then, we can transform the original problems into the one defined on rectangular domain \cite{Hao-Zhang-Du-2021}.
		
		Figure \ref{VO_Lap_Coexistence_dynamics} shows the dynamics of the coexistence of anomalous diffusion equation with different $\alpha(x)$. Here, we set the rectangular domain $\Omega = [-1,1]^2$, the diffusion coefficient $\kappa = 0.2$ and initial condition $u(x,0)=1$. For time discretization we use Crank-Nicolson scheme with $h=2^{-8} $ and $ \Delta t = 2\times10^{-4}$. We find that the smaller the order $\alpha(x)$, the slower the rate of diffusion (see the left and middle columns in Figure \ref{VO_Lap_Coexistence_dynamics}). However, when the order $\alpha(x)$ exhibits significantly heterogeneous at both ends of the center point, the large power side represents faster diffusion rate than the other side and even weaken it (see Figure \ref{VO_Lap_Coexistence_dynamics} for the right column).
		\begin{figure}[h]
			\centering
				\centering
				\includegraphics[width=0.5\linewidth]{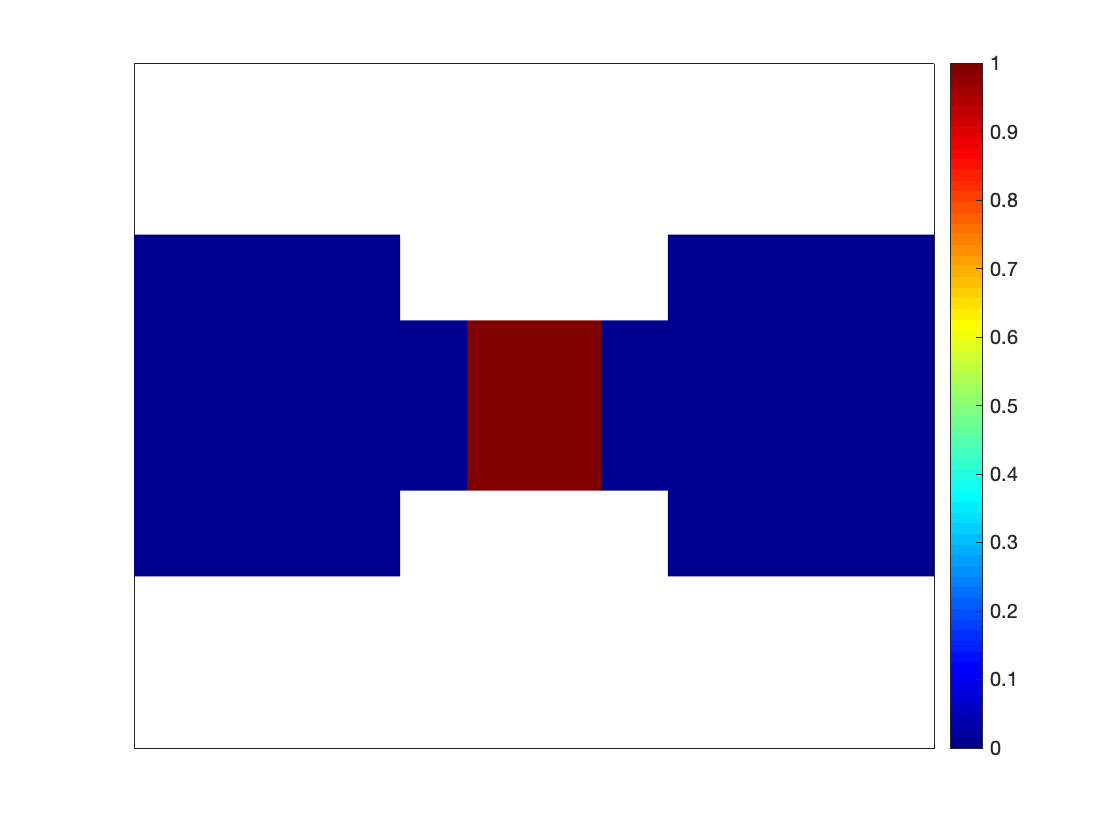}
				\caption[Figure3]{
					The profile of  solution at the initial time (Example \ref{Ex-general-domain}).}
				\label{VO_Lap_Coexistence_initial}
		\end{figure}
		
		\begin{figure}[h]
			\centering
			\subfigure[$t=0.2$]{ 
				\begin{minipage}[t]{0.3\linewidth}
					\centering
					\includegraphics[width=1\linewidth]{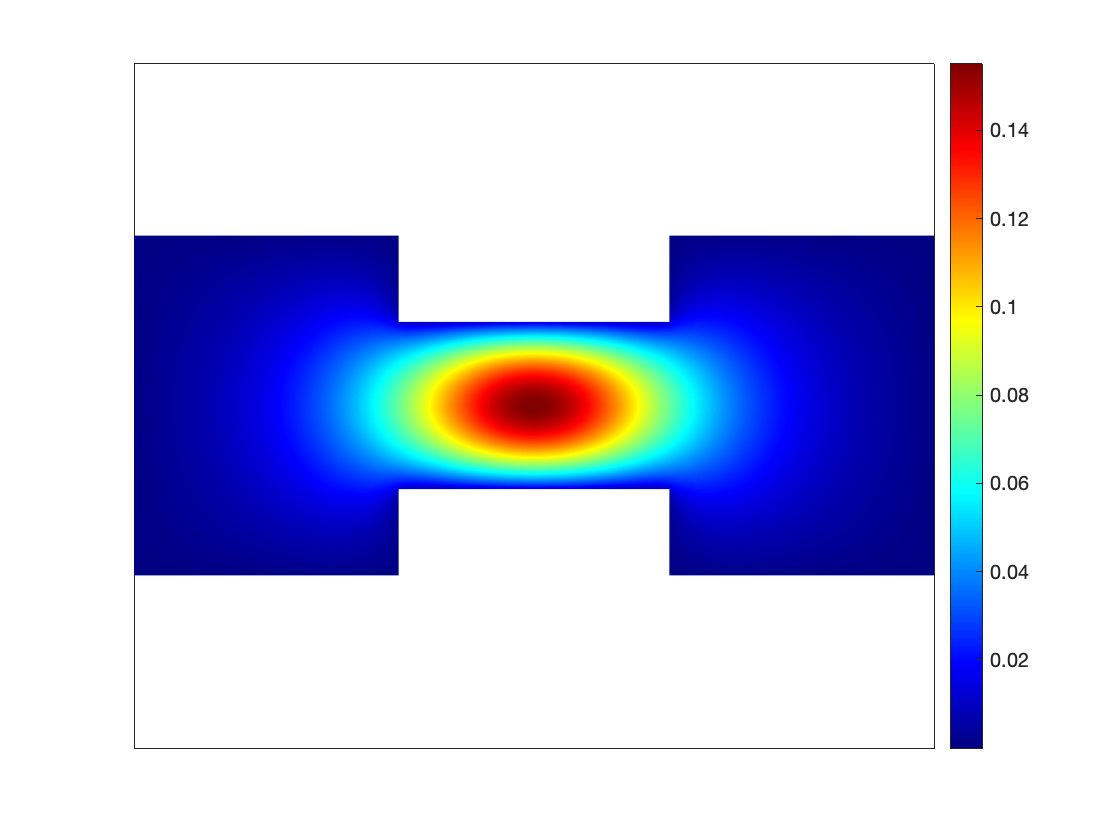}
				\end{minipage}
			}
			\subfigure[$t=0.2$]{
				\begin{minipage}[t]{0.3\linewidth}
					\centering
					\includegraphics[width=1\linewidth]{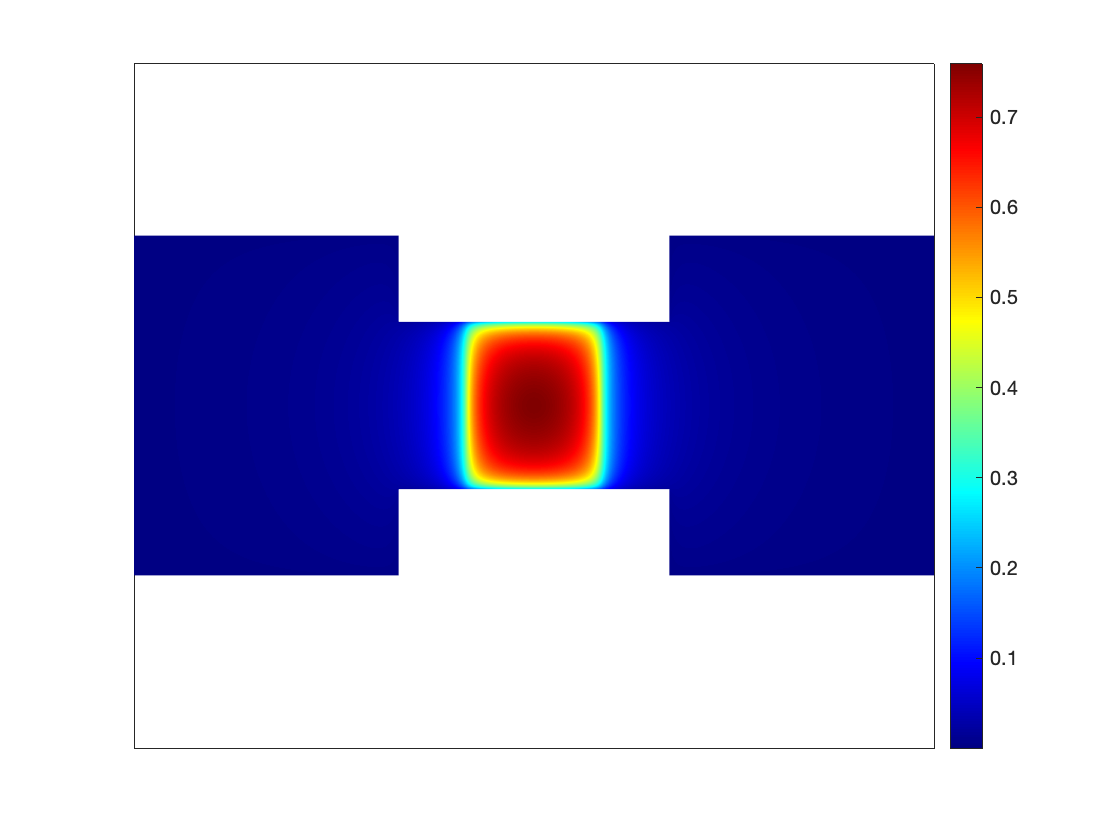}
				\end{minipage}
			}
			\subfigure[$t=0.2$]{
				\begin{minipage}[t]{0.3\linewidth}
					\centering
					\includegraphics[width=1\linewidth]{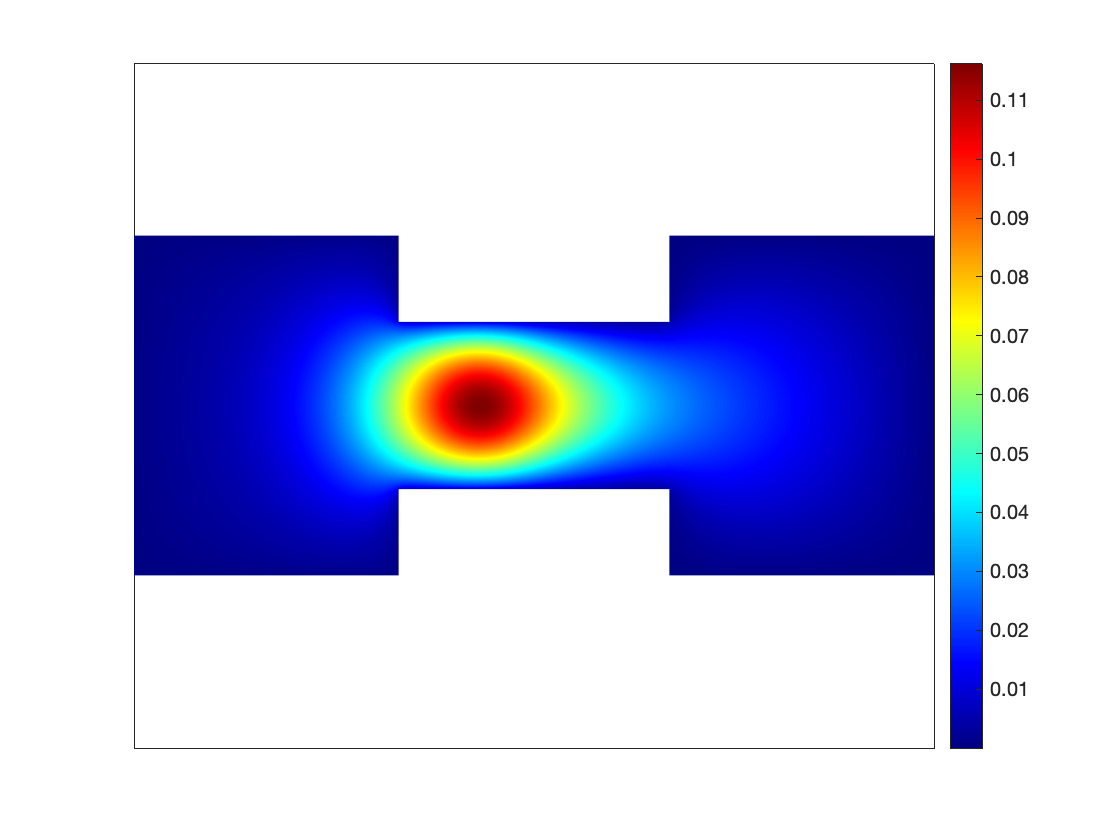}
				\end{minipage}
			}
			\subfigure[$t=0.5$]{
				\begin{minipage}[t]{0.3\linewidth}
					\centering
					\includegraphics[width=1\linewidth]{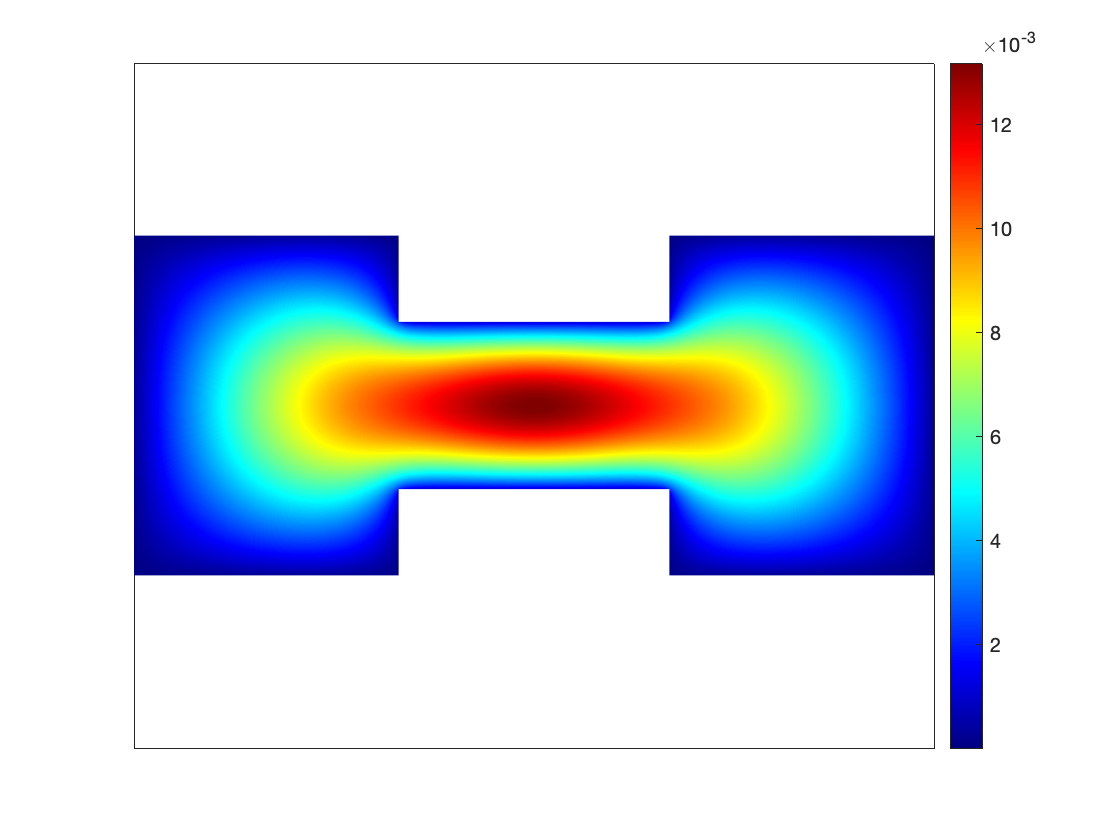}
				\end{minipage}
			}
			\subfigure[$t=0.5$]{
				\begin{minipage}[t]{0.3\linewidth}
					\centering
					\includegraphics[width=1\linewidth]{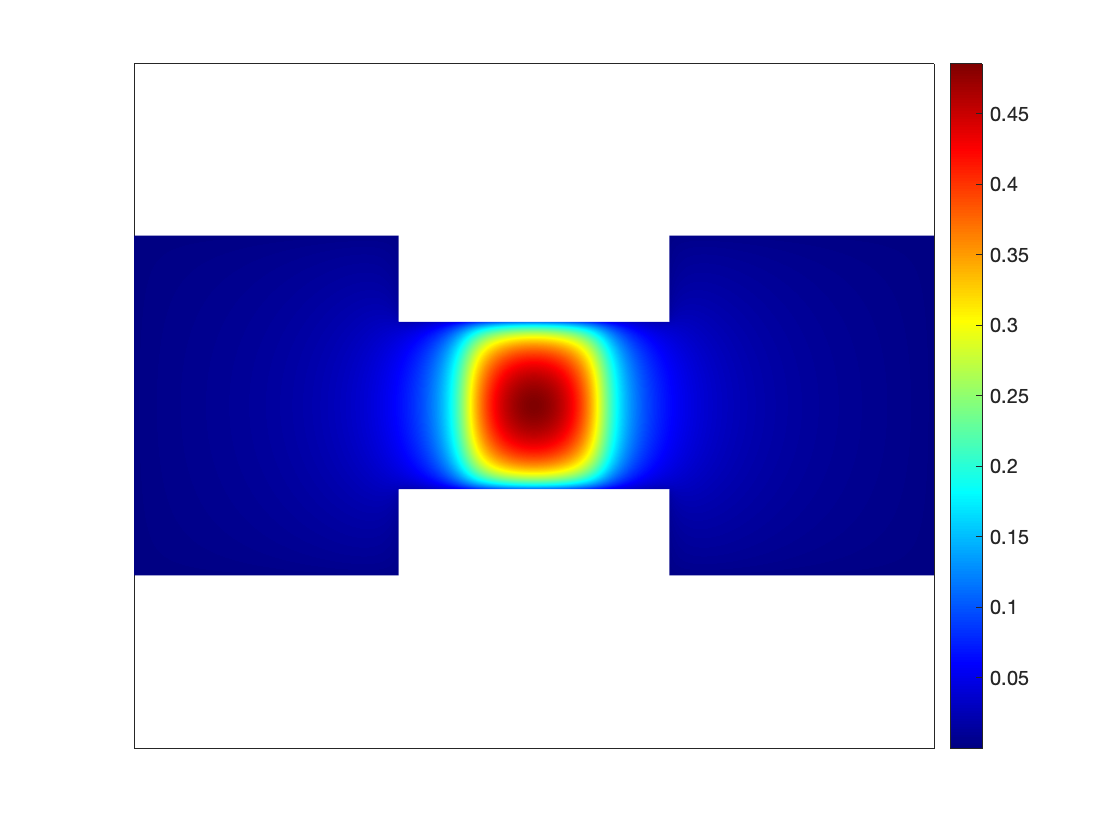}
				\end{minipage}
			}
			\subfigure[$t=0.5$]{
				\begin{minipage}[t]{0.3\linewidth}
					\centering
					\includegraphics[width=1\linewidth]{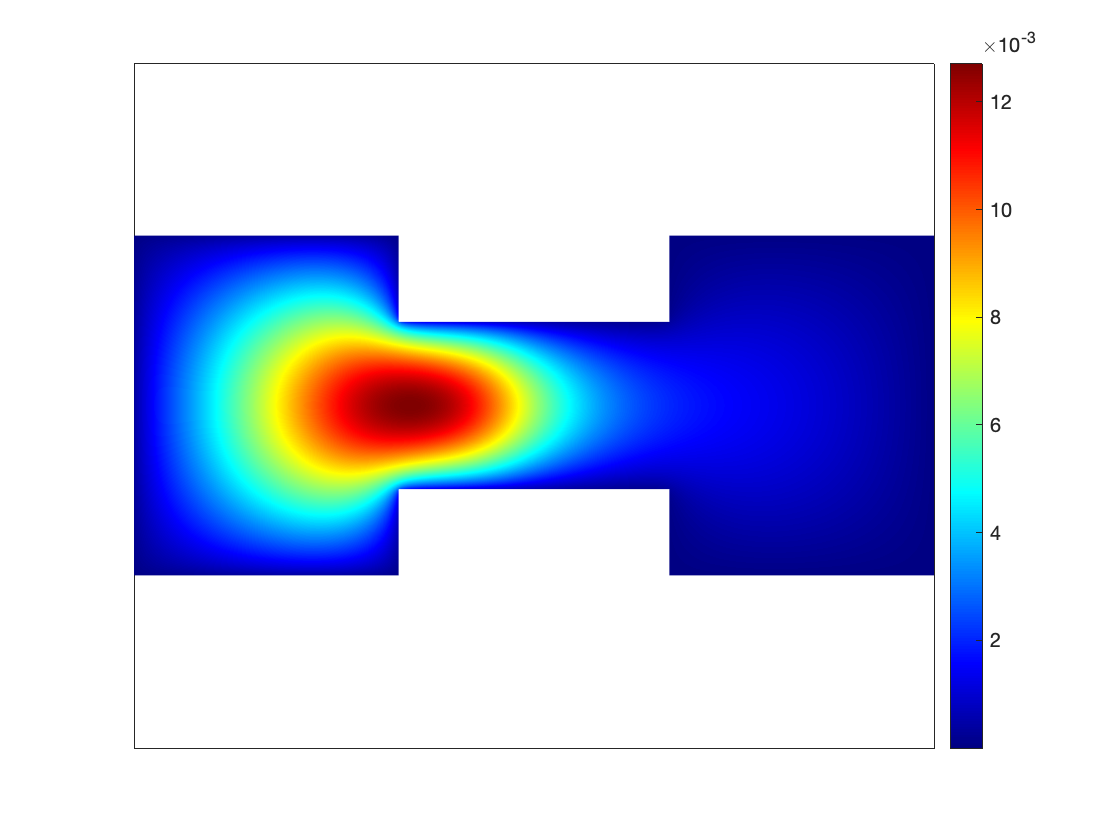}
				\end{minipage}
			}
			\subfigure[$t=1$]{
				\begin{minipage}[t]{0.3\linewidth}
					\centering
					\includegraphics[width=1\linewidth]{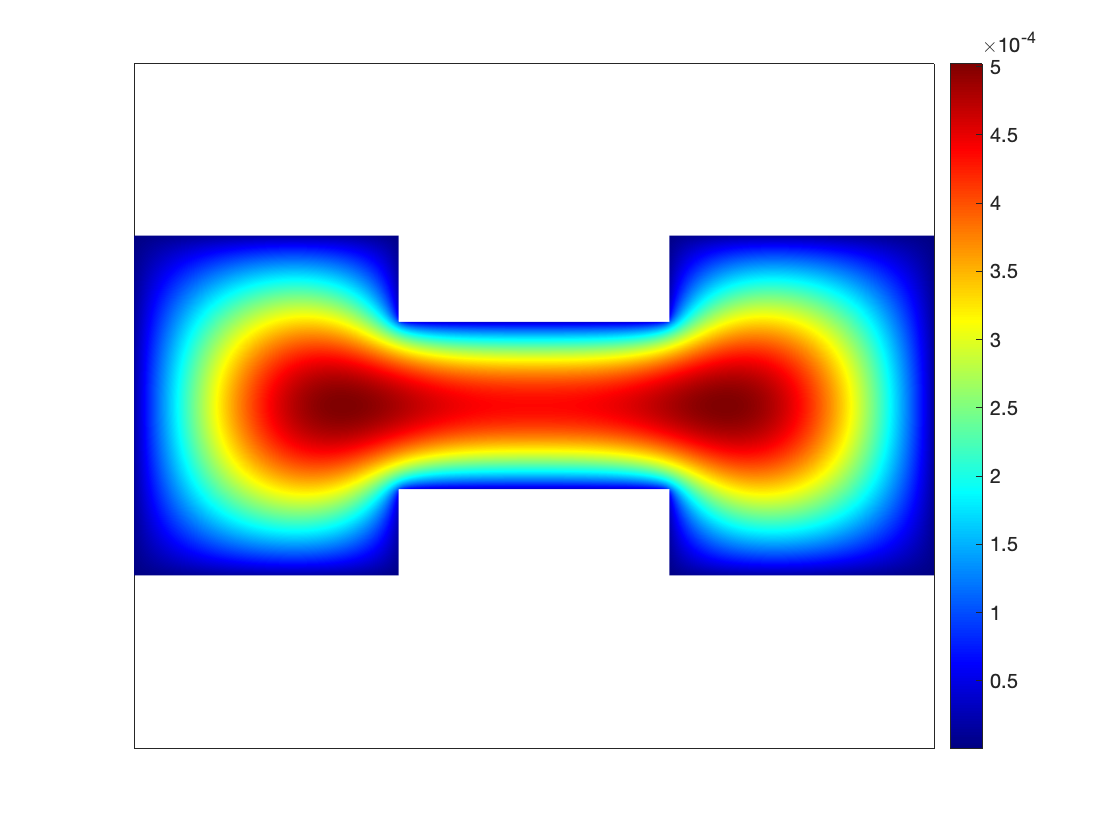}
				\end{minipage}
			}
			\subfigure[$t=1$]{
				\begin{minipage}[t]{0.3\linewidth}
					\centering
					\includegraphics[width=1\linewidth]{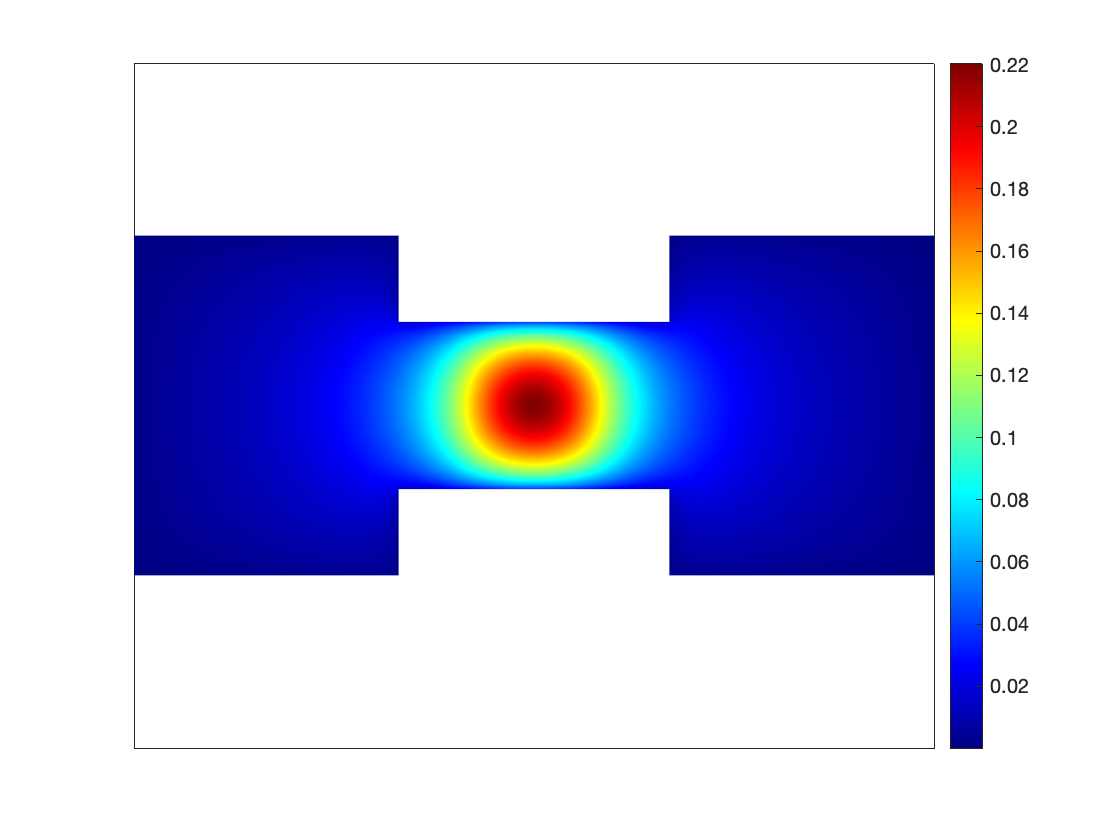}
				\end{minipage}
			}
			\subfigure[$t=1$]{
				\begin{minipage}[t]{0.3\linewidth}
					\centering
					\includegraphics[width=1\linewidth]{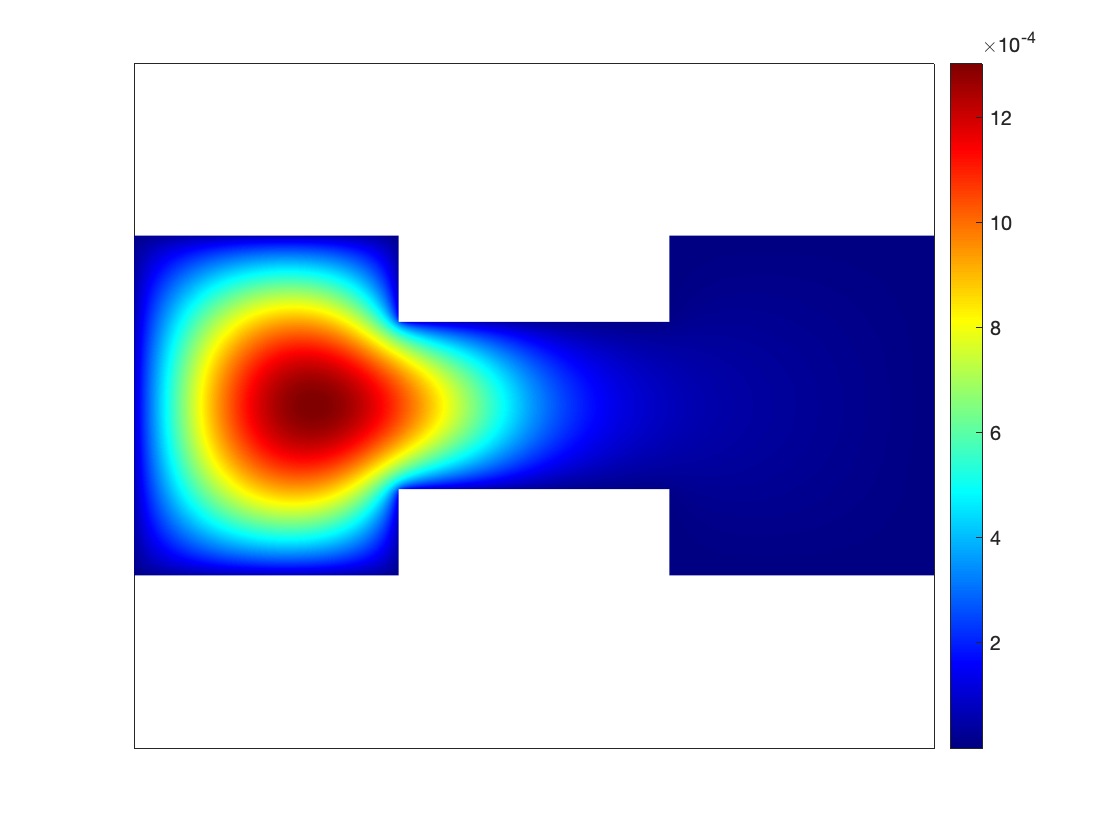}
				\end{minipage}
			}
			\subfigure[$t=2$]{
				\begin{minipage}[t]{0.3\linewidth}
					\centering
					\includegraphics[width=1\linewidth]{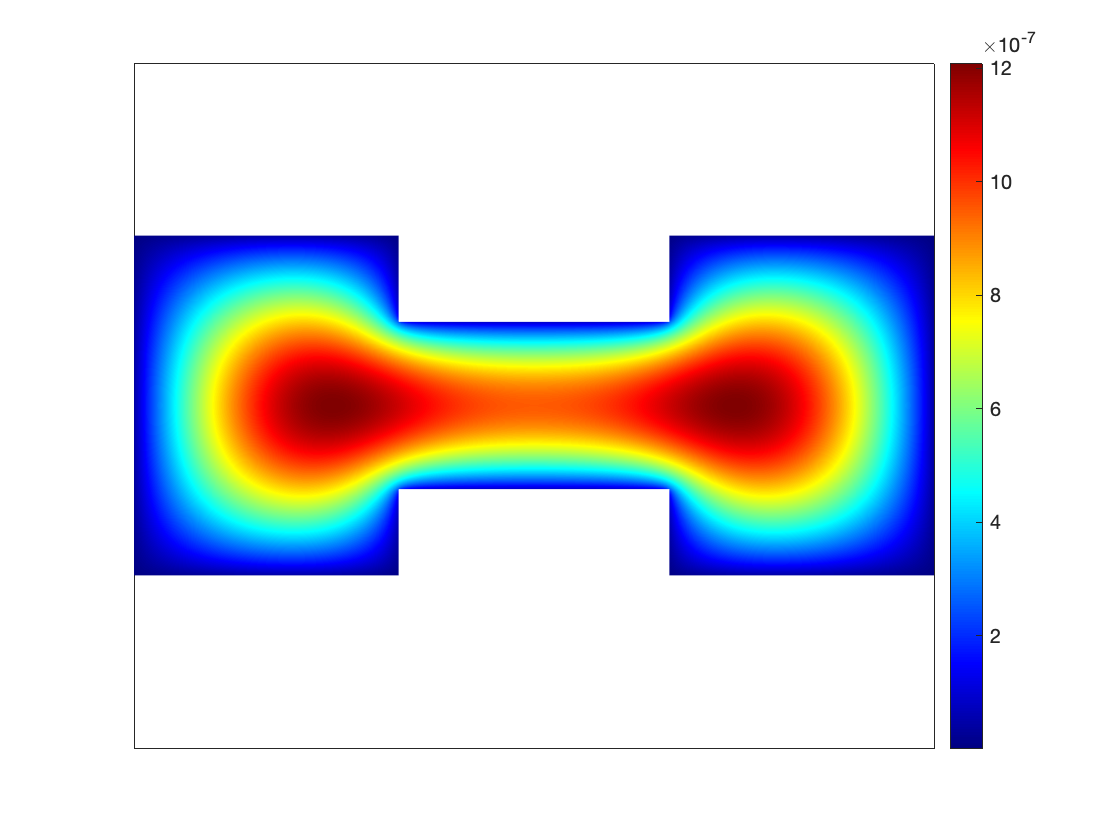}
				\end{minipage}
			}
			\subfigure[$t=2$]{
				\begin{minipage}[t]{0.3\linewidth}
					\centering
					\includegraphics[width=1\linewidth]{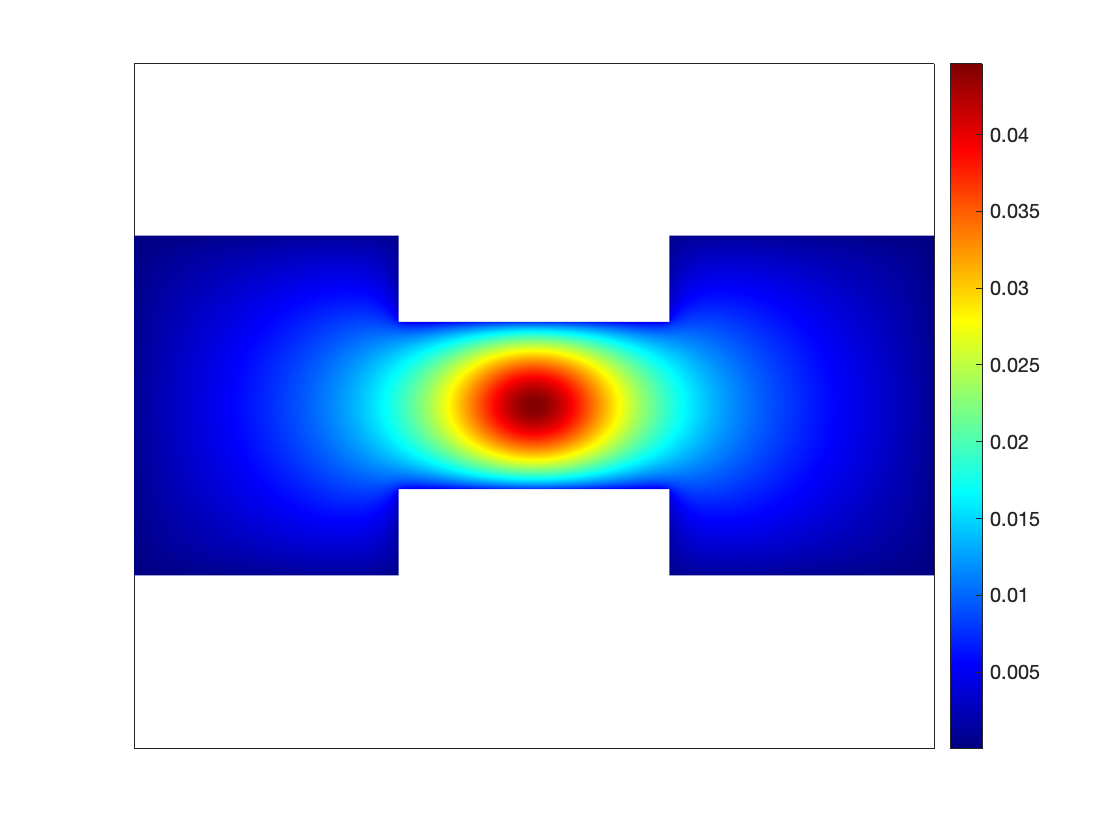}
				\end{minipage}
			}
			\subfigure[$t=2$]{
				\begin{minipage}[t]{0.3\linewidth}
					\centering
					\includegraphics[width=1\linewidth]{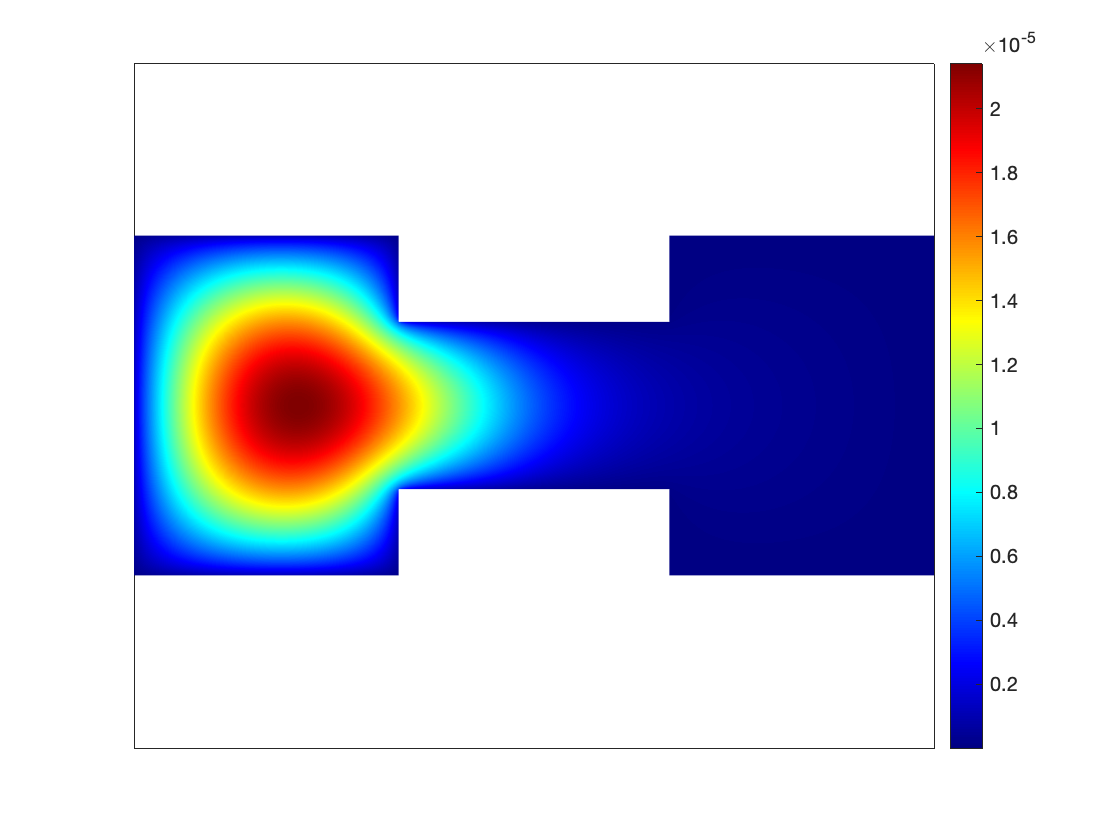}
				\end{minipage}
			}
			\caption{Dynamics of the coexistence of anomalous diffusion equation(left: $\alpha(x)=1.5+|x|/4; $ middle: $\alpha(x)=0.6+|x|/2;$ right: $\alpha(x)=0.3\tanh(10x_1)+1.7$. (Example \ref{Ex-general-domain}).}
			\label{VO_Lap_Coexistence_dynamics}
		\end{figure}
		
		\begin{exm}[Application to the time-dependent problem in 3D]
			\label{parabolic-eq-3-D}
			Consider the 3D time-dependent variable-order fractional diffusion problem  \cite{Minden-Ying-2020}:
			\begin{eqnarray}
				&&\partial_t u(x, t)=-(-\Delta)^{\alpha(x)/2}u(x,t) + f(x,t),  \quad \, x \in \Omega, \, t>0, \nonumber\\
				&&u(x, t)=0,  \quad \, x \in \mathbb{R}^3 \backslash \Omega,\,  t \geq 0,\nonumber \\
				&&u(x,t) = u_0(x), \quad  \, x \in \Omega.\nonumber
			\end{eqnarray}
		\end{exm}
		
		%
		
		Here, the right hand side is set as $f(x) = 0$ with $\Omega = [-1,1]^3$ and the initial condition
		\begin{eqnarray}
			&&u_0(\mathbf{x})=\prod_{i=1}^{3}\frac{1}{4}(1+\cos(2\pi \upsilon_ix_i-\pi))^2,\nonumber
		\end{eqnarray}
		with $\upsilon_1=3, \upsilon_2=11, \upsilon_3=2$. We use the second-order finite difference approximation for spatial discretization and  a Crank-Nicolson method for the temporal discretization,  respectively. To demonstrate the efficiency of our methods, we present  the CPU time $t$ of solving the 3D linear system of equations $Ax=b$ by the BiCGSTAB method and corresponding BiCGSTAB iterations $n$ for a single time step with different $\alpha(x)$ in Table \ref{table8}. For convenience, at the bottom of Table \ref{table8} we give an estimate of the asymptotic decay rate of the error as $N$ increases, which is achieved by a least-squares fit of the log-error to log-$N$. It shows that for fixed $N$ and $\Delta t$, the time $t$ gets longer for larger $\alpha(x)$, because the stiffness matrix from larger $\alpha(x)$ has bigger conditional number and it will affect the BiCGSTAB iterations. Furthermore, the corresponding timing results are plotted in Figure \ref{VO_Lap_Time_3D_runtime}.
  \begin{table}[h]
        \centering
        \caption{Runtime $t$ in seconds and number of iterations $n$ required to compute 3-D time-dependent variable fractional diffusion problem for a single time step by BiCGSTAB method. (Example \ref{parabolic-eq-3-D}).}
		\setlength{\tabcolsep}{4pt}{
				\begin{tabular}{llllllllll} 
					\toprule[1pt]
					\multirow{2}{*}{$N^3$} &\multirow{2}{*}{$\Delta t$}  & \multicolumn{2}{l}{$\alpha(x) = 1-0.5\tanh (|x|)$}  & \multicolumn{2}{l}{$\alpha(x) = 1+|x|/4$} & \multicolumn{2}{l}{$\alpha(x) = 1.5+|x|/4$}& \multicolumn{2}{l}{$\alpha(x) = 1.6$}\\
					\cmidrule(r){3-4}  
					\cmidrule(r){5-6}
					\cmidrule(r){7-8}
					\cmidrule(r){9-10}
					& &$t$ & $n$ & $t$ &$n$ & $t$ &$n$& $t$ &$n$\\
					\midrule[0.25pt]
					$31^3$ & 1/32 &4.78e+01 & 13&5.33e+01&38 & 7.78e+01&94&6.31e-01 &61 \\
					$63^3$ & 1/64 &9.94e+01& 13 &1.86e+02&47 &4.58e+02&158 &6.82e+00 & 86\\
					$127^3$ & 1/128&6.23e+02&14  &1.70e+03&55 & 6.67e+03&243 &8.70e+01 &116\\
					$255^3$& 1/256&5.20e+03&14&1.77e+04 &63& 8.05e+04 &330 &1.03e+03 &153\\
					\midrule[0.25pt]
					\multicolumn{2}{l}{Rate}& 2.26&*&2.79&*& 3.35&* & 3.52&*\\
					\bottomrule[1pt]
			\end{tabular}}
			\label{table8}
		\end{table}
		\begin{figure}[h]
			\centering
			\centering
			\includegraphics[width=0.6\linewidth]{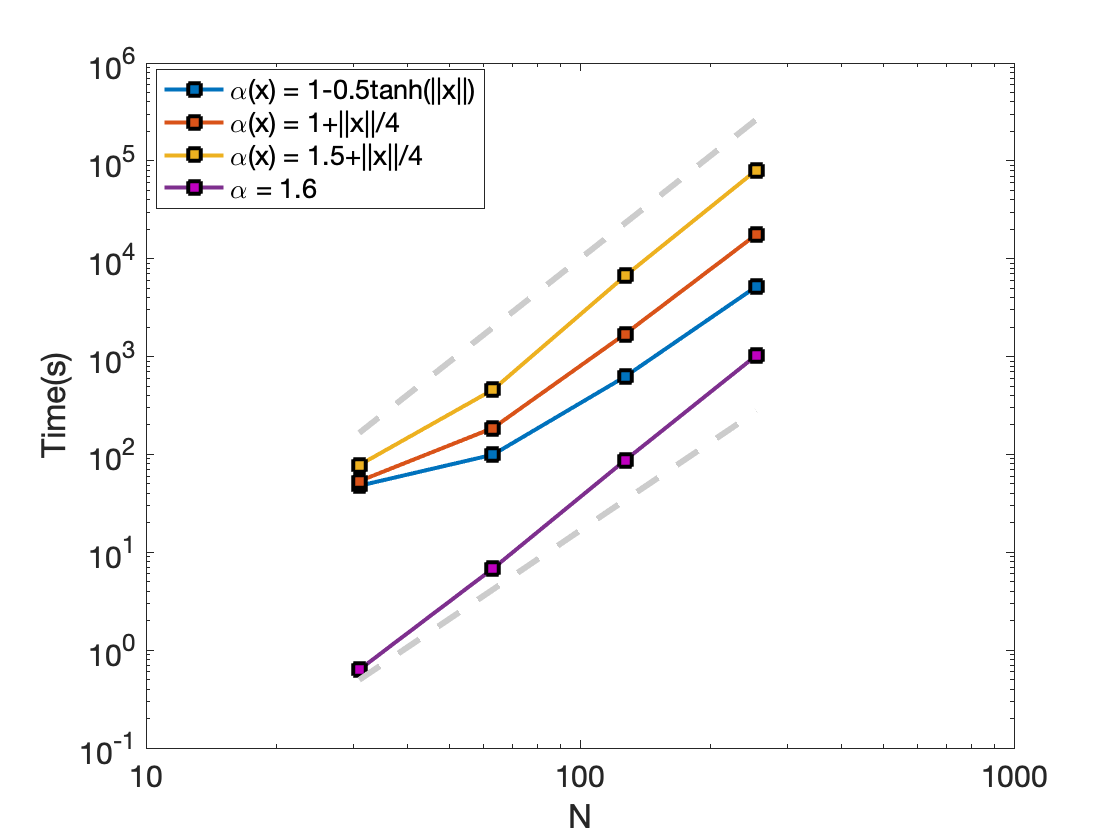}
			\caption[Figure3]{For 3D Example \ref{parabolic-eq-3-D}, we plot the runtime $t$ required for BiCGSTAB method to attain an accuracy of $\epsilon = 10^{-12}$ as in Table \ref{table8}. The top trend line is our computational time complexity $O(N^{3.5})$ and the bottom gray line is for the optimal complexity $O(N^3)$ (Here $N$ denotes the number of the grids in each spatial direction). }
			\label{VO_Lap_Time_3D_runtime} 
		\end{figure}

	\section{Proofs}\label{sec-proofs}	
	In this section, we present the proofs of conclusions in Sections \ref{sec-FDM-approx} and 3.
 
\subsection{Proof of Theorem \ref{prop-equivalence-definitions}}
  \begin{proof}
		By the definition of variable-order  fractional Laplacian \eqref{def-v-int-frac-lap}, we can rewrite it as follows
		\begin{eqnarray}
			(-\Delta)^{\alpha(x)/2} u= \frac{1}{2} c_{d,\alpha(x)}  \int_{\mathbb{R}^d} \frac{ 2u(x)-u(x+y)-u(x-y)} {|y|^{d+\alpha(x)}}\text{d}y, \quad \forall x\in \mathbb{R}^d. 
		\end{eqnarray}
  It remains to show that the definition of  variable-order  fractional Laplacian via Fourier transform
  \eqref{def-v-frac-lap} can be also written  in this form. 
  
		Recalling the identity (e.g., see \cite{bucur2016nonlocal})
		\begin{eqnarray}
			c_{d,\alpha(x)} \int_{\mathbb{R}^d} \frac{1- \cos(\xi \cdot y)}{|y|^{d+\alpha(x)}} \text{d}y = |\xi|^{\alpha(x)},
		\end{eqnarray}
		we have 	
		\begin{eqnarray}
			|\xi|^{\alpha(x)} \cdot	\mathcal{F} [u](\xi) &=&	c_{d,\alpha(x)} \int_{\mathbb{R}^d} \frac{1- \cos(\xi \cdot y)}{|y|^{d+\alpha(x)}} \text{d}y  \cdot	\mathcal{F} [u] (\xi) \notag\\
			&=&  \frac{1}{2} c_{d,\alpha(x)} \int_{ \mathbb{R}^d} \frac{2-e^{\mathbf{i} \xi \cdot y} -e^{-\mathbf{i} \xi \cdot y}}{|y|^{d+\alpha(x)}} \text{d}y  \cdot	\mathcal{F} [u] (\xi) \notag\\
			&=&  \frac{1}{2} c_{d,\alpha(x)} \int_{ \mathbb{R}^d} \frac{ \mathcal{F} \big[ 2u(x)-u(x-y)-u(x+y) \big]  (\xi)}{|y|^{d+\alpha(x)}} \text{d}y  .
		\end{eqnarray}
		Taking the inverse Fourier transform on both sides leads to 
		\begin{eqnarray}
			\mathcal{F}^{-1} \bigg( 	|\xi|^{\alpha(x)} \cdot	\mathcal{F} [u](\xi) \bigg)
			&=&  \frac{1}{2} c_{d,\alpha(x)} 	\mathcal{F}^{-1} \bigg(\int_{ \mathbb{R}^d} \frac{ \mathcal{F} \big[ 2u(x)-u(x-y)-u(x+y) \big]  (\xi)}{|y|^{d+\alpha(x)}} \text{d}y \bigg) \notag\\
			&=&  \frac{1}{2} c_{d,\alpha(x)} 	\int_{ \mathbb{R}^d}  \frac{\mathcal{F}^{-1}  \bigg( \mathcal{F}  \big[ 2u(x)-u(x-y)-u(x+y) \big]  (\xi) \bigg)}{|y|^{d+\alpha(x)}} \text{d}y  
			\notag\\
			&=&   \frac{1}{2} c_{d,\alpha(x)}  \int_{\mathbb{R}^d} \frac{ 2u(x)-u(x+y)-u(x-y)} {|y|^{d+\alpha(x)}}\text{d}y,
		\end{eqnarray}
		where we have used Fubini's theorem for the second equal sign. 
		This completes the proof. 
	\end{proof}

\subsection{Proof of Theorem \ref{thm-approximation-property}}

In this subsection, we present the proof for  the second-order approximation in Theorem \ref{thm-approximation-property}.  
 The semi-discrete Fourier transform is closely connected with the continuous Fourier transform.  The following lemma is needed in our proof. 
	\begin{lem}[Poisson summation formula]
		Let $u \in L^1 (\mathbb{R}^d) $ be such that its Fourier transform $ \mathcal{F}[u]$ is also absolutely integrable, and $u $ satisfies the two estimates $u(x) = \mathcal{ O} ((1+|x|) ^{-d-\varepsilon})$ and $ \mathcal{F}[u] (\xi) = \mathcal{O} ( (1+|\xi|)^{-d-\varepsilon})   $ for positive $\varepsilon $. Then we have the identity 
		\begin{equation*}
			\mathcal{F}_h [u_h] (\xi)= \sum_{\eta \in \frac{2\pi}{h}\mathbb{Z}^d} \mathcal{F}[u](\xi +  \eta), \quad \xi \in D_h=\big[ -\frac{\pi}{h}, \frac{\pi}{h}\big]^d. 	
		\end{equation*}	
		
	\end{lem}
	\begin{proof}
		Recall the Poisson summation formula (e.g., see Chapter 4 in \cite{Pinsky-2002}.)   
		\begin{equation}
			\sum_{j\in \mathbb{Z}^d} u(j) e^{-\mathbf{i} \xi\cdot j}= \sum_{j\in \mathbb{Z}^d} \mathcal{F}[u]( \xi + 2\pi j), \quad  \xi \in [-\pi,\pi]^d. \label{eq-possion-summation-formula}
		\end{equation}
		The convergence of the sums is  in $L^1 ([-\pi,\pi]^d)$. If $u $ satisfies the two estimates $|u(x)| = \mathcal{ O} ((1+|x|) ^{-d-\varepsilon})$ and $| \mathcal{F}[u] (\xi)| = \mathcal{O} ( (1+|\xi|)^{-d-\varepsilon})   $ for positive $\varepsilon $, then the two sums above are absolutely convergent and their limiting functions are continuous. Therefore the above identity holds pointwise.
		By the scaling property of the Fourier transform, we have
		\begin{equation*}
			\sum_{j h\in h\mathbb{Z}^d} u(jh) e^{-\mathbf{i} \xi\cdot jh}=	\sum_{j\in \mathbb{Z}^d} u(jh) e^{-\mathbf{i} \xi\cdot jh}= \frac{1}{h^d}\sum_{j\in \mathbb{Z}^d} \mathcal{F}[u]\bigg( \frac{\xi}{h} + \frac{2\pi}{h} j\bigg).  
		\end{equation*}
		By the definition of the semi-discrete transform \eqref{def-discrete-fouier-transform}, we can obtain the desired result. 
	\end{proof}
	The Poisson summation formula is essential   to prove the convergence of our discrete approximation.
By Taylor expansion,   it is not hard to prove  the following estimates for $M_h(\xi)$ in \eqref{def-multiplier}:
	\begin{eqnarray}
		&& \big(M_h(\xi)\big) ^{\alpha(x)/2} \leq c|\xi|^{\alpha(x)}, \label{def-multiplier-boundeness} \\
		&&  | M_h(\xi) ^{\alpha(x)/2} - |\xi|^{\alpha(x)}|  \leq ch^2    \label{def-multiplier-equivalence}
	\end{eqnarray}
	for $ \xi\in D_h$ and sufficiently small $h$.     Here and throughout the following, $c$ is a  constant independent of the step-size $h$ and the function $u$ but may change from line to line.

	\begin{proof}
		For 	$u \in  \mathcal{B}^{s+\alpha_{\max}} (\mathbb{R}^d)$ with $s \leq 2$, it is clear that $  u	\in C(\mathbb{R}^d) $ and the grid function $u_h(x)$ is well defined. 		We use the semi-Fourier analysis technique to estimate the convergence of the discrete operator.  
		For $x\in \Omega_h$, 
		\begin{eqnarray}
			&&	(-\Delta _h)^{\alpha(x)/2} u(x) -(-\Delta )^{\alpha(x)/2} u(x)  \notag\\
			& =& \frac{1}{(2\pi)^{d}} \int_{D_h} e^{\mathbf{i }\xi \cdot x }  M_h(\xi) ^{ \alpha(x)/2}  \mathcal{F}_h [u_h](\xi) \text{d}\xi- \frac{1}{(2\pi)^d } \int_{\mathbb{R}^d} e^{ \mathbf{i} \xi\cdot x }  |\xi|^{\alpha(x)} \mathcal{F} [u] (\xi) \text{d}\xi  \notag\\
			&:=& I_1+ I_2 +I_3 \label{def-sum-of-four-iterms},
		\end{eqnarray}
		where each integral is defined below
		\begin{eqnarray*}
			&& 	I_1 := 	\frac{1}{(2\pi)^{d}} \int_{D_h} e^{\mathbf{i }\xi \cdot x }  M_h(\xi) ^{\alpha(x)/2}  \big( \mathcal{F}_h [u_h](\xi) - \mathcal{F} [u] (\xi) \big ) \text{d}\xi , \\
			&& I_2 := \frac{1}{(2\pi)^{d}} \int_{D_h} e^{\mathbf{i }\xi \cdot x }  \big( M_h(\xi) ^{ \alpha(x)/2}  - |\xi|^{\alpha(x)}\big) \mathcal{F} [u](\xi) \text{d}\xi,
			\\
			&& I_3 := -  \frac{1}{(2\pi)^d } \int_{D_h^c} e^{ \mathbf{i} \xi\cdot x }  |\xi|^{\alpha(x)} \mathcal{F} [u] (\xi) \text{d}\xi .
		\end{eqnarray*}
		We estimate each term in \eqref{def-sum-of-four-iterms}.
		For the first item $I_1$, we need to estimate the difference  $\mathcal{F}_{h} [u_h] (\xi) -F[u] (\xi) $ in the integrand.  Noting the assumption $u\in \mathcal{B}^{s+\alpha_{\max}} (\mathbb{R}^d)$ and  by the Poisson summation formula, we have 
		\begin{eqnarray*}
			\abs{\mathcal{F}_h [u_h] (\xi) - 	\mathcal{F} [u] (\xi) }&=& \abs{\sum_{\eta \in \frac{2\pi}{h}\mathbb{Z}^d/ \{0\}} \mathcal{F}[u](\xi +  \eta)} \notag \\
			&\leq & \max_{\eta \in \frac{2\pi}{h}\mathbb{Z}^d/ \{0\}} |\xi+\eta|^{-s}  \cdot  \sum_{\eta \in \frac{2\pi}{h}\mathbb{Z}^d/ \{0\}} |\xi+\eta|^{s} \abs{\mathcal{F}[u](\xi +  \eta)} \notag \\
			&\leq & \frac{h^{s}}{\pi^{s}}    \sum_{\eta \in \frac{2\pi}{h}\mathbb{Z}^d/ \{0\}} |\xi+\eta|^{s} |\mathcal{F}[u](\xi +  \eta)|.  
		\end{eqnarray*}
		Plugging in $I_1$ and using the equivalence \eqref{def-multiplier-equivalence},  we have 
		\begin{eqnarray*}
			\abs{I_1} & \leq & c h^{s} \int_{D_h}  |\xi| ^{\alpha (x) }  \bigg(\sum_{\eta \in \frac{2\pi}{h}\mathbb{Z}^d/ \{0\}} |\xi+\eta|^{s} |\mathcal{F}[u](\xi +  \eta)| \bigg ) \text{d}\xi  \notag \\
			&\leq &  c h^{s} \int_{\mathbb{R}^d}  |\xi| ^{\alpha (x) +s}  |\mathcal{F}[u](\xi)| \text{d}\xi  \notag\\
			&\leq &  c h^{s}  \int_{|\xi|\leq 1}  |\xi| ^{\alpha (x) +s}  |\mathcal{F}[u](\xi) | d\xi +c h^{s}\int_{|\xi|>1 }  |\xi| ^{\alpha (x) +s}  |\mathcal{F}[u](\xi)| \text{d}\xi 	
			\notag\\
			&\leq &  c h^{s} \int_{|\xi|\leq 1}  |\mathcal{F}[u](\xi) |d\xi +c h^{s}
   \int_{|\xi|>1 }  |\xi| ^{\alpha_{\max}  +s}  |\mathcal{F}[u](\xi)| \text{d}\xi 
			\notag\\
			&\leq &  c h^{s} \|u\|_{ \mathcal{B}^{\alpha_{\max} +s}}	.
		\end{eqnarray*}
		For the second term similarly we have 
		\begin{eqnarray*}
			\abs{I_2} := \frac{1}{(2\pi)^{d}} \abs{\int_{D_h} e^{\mathbf{i }\xi \cdot x }  \big( M_h(\xi) ^{ \frac{\alpha (x) }{2}}  - |\xi|^{\alpha(x)}\big) \mathcal{F} [u](\xi) \text{d}\xi} \leq c h^2   \int_{\mathbb{R}^d}  |\xi| ^{\alpha (x) }  |\mathcal{F}[u](\xi)| \text{d}\xi  \leq  c h^{s} \|u\|_{ \mathcal{B}^{\alpha_{\max} +s}}.
		\end{eqnarray*} 
		For the third term, we have 
		\begin{eqnarray*}
			\abs{I_3}  &=&  \abs{- \frac{1}{(2\pi)^d } \int_{D_h^c} e^{ \mathbf{i} \xi\cdot x }  |\xi|^{\alpha(x) +s } |\xi|^{-s} \mathcal{F} [u] (\xi) \text{d}\xi}   \\
			&\leq& \frac{h^{s}}{{(2\pi)}^{d} {\pi^s}} \int_{D_h^c}  |\xi|^{\alpha(x) +s }  |\mathcal{F} [u] (\xi)| \text{d}\xi   \\
   &\leq& c h^{s}   \int_{\mathbb{R}^d}  |\xi| ^{\alpha (x) +s} | \mathcal{F}[u](\xi) |\text{d}\xi  
    \leq   c h^{s} \|u\|_{ \mathcal{B}^{\alpha_{\max} +s}}.
		\end{eqnarray*}
		Combining the estimates of $I_i$'s together leads to the desired result.  
	\end{proof}

  \subsection{Proof of Theorem \ref{thm-stability-convergence}}
		Here we  analyze the  stability and convergence  of the  finite difference scheme  for the fractional elliptic equation  on the unit interval $\Omega=(0,1)$.  For  simplicity, we assume that the variable coefficient $b>0$.

		The stiffness matrix associated with the discrete fractional Laplacian operator with the variable-order is	
		\[
		S^{(\mathbf{\alpha})} :=
		\left[ {\begin{array}{cccc}
				a^{(\alpha_1)}_{0}/h^{\alpha_1} & a^{(\alpha_1)}_{1}/h^{\alpha_1} & \cdots &a^{(\alpha_1)}_{N-2}/h^{\alpha_1}\\
				a^{(\alpha_2)}_{-1}/h^{\alpha_2} & a^{(\alpha_2)}_{0}/h^{\alpha_2} & \cdots &a^{(\alpha_2)}_{N-3}/h^{\alpha_2}\\
				\vdots & \vdots & \ddots & \vdots\\
				a^{(\alpha_{N-1})}_{2-N}/h^{\alpha_{N-1}} & a^{(\alpha_{N-1})}_{3-N}/h^{\alpha_{N-1}} & \cdots &a^{(\alpha_{N-1})}_{0}/h^{\alpha_{N-1}}\\
		\end{array} } \right]_{(N-1)\times (N-1)}.
		\]
		

		
		We use the maximum  value principle  to analyze the stability of the scheme. To this end, we need to examine the properties of the coefficients in the discrete fractional Laplacian defined in \eqref{def-quasi-convolution-approximation}. 
		
		The properties of the coefficients are stated in the following lemma. 
		\begin{lem}\label{lem-properties-coefficients}
			For the coefficients defined in \eqref{def-a_k_alpha}, we have the following properties. 
			\begin{enumerate}
				\item For fixed  $\alpha_j \in (0,2)$,  $\alpha_0^{(\alpha_j)}>0$ and $a_n^{(\alpha_j)}= a_{-n}^{(\alpha_j)}<0$ for $n\neq 0,\ j=1,\cdots, N-1$. 
				
				\item The summation of the coefficients equal to zero, that is  $\sum_{n}a_n^{(\alpha_j)} =0$. 
				
				\item  We have the estimates 
				\begin{equation}
					\frac{c_1}{ |n|^{\alpha_j+1}}\leq 	|a_{n}^{(\alpha_j)}|\leq   \frac{c_2}{ |n|^{\alpha_j+1}},\quad \sum_{n\geq N} |a_{n}^{(\alpha_j)}| \geq c_3 \frac{1}{N^{\alpha_j} }.
				\end{equation}
			\end{enumerate}
		\end{lem}
		\begin{proof}
			For proof, the interested readers may refer to \cite{Sun-Gao-2020}.
		\end{proof}

		\begin{lem}
  \label{lem:barrier}
			Let $A$ be a monotone matrix  of order $n$ and $v$ is a normalized vector, $\|v\|_{l^\infty}:= \max_{j}\{v_j\}=1$ such that $\min_{j} (Av)_j \geq \beta$ for some positive scalar $\beta$. Then we have 
			$$ \|A^{-1}\|_{l^\infty} \leq 1/\beta .$$
		\end{lem}

		With the above lemma,	we are ready to present  the stability and convergence  of the difference scheme \eqref{v-order-scheme}-\eqref{v-order-scheme-bc}.

		\begin{proof}
			Denote $U=(u_1,u_2,\cdots, u_{N-1})^T$ and $F=(f_1,f_2,\cdots, f_{N-1})^T$. Then the difference scheme can be rewritten as the matrix form $S^{(\mathbf{\alpha})} U=F $. By Lemma \eqref{lem-properties-coefficients}, we know that the matrix $S^{(\mathbf{\alpha})}$ is diagonally dominant and monotone matrix. Thus $S^{(\mathbf{\alpha})}$ is not singular matrix and we can write $ U= (S^{(\mathbf{\alpha})})^{-1} F $.  Thus $\|U\|_{l^\infty} \leq  \|(S^{( \mathbf{\alpha})})^{-1} \|_{l^\infty}  \|F\|_{l^\infty}$. 
			Take the constant vector $v=(1,1,\cdots,1)^T$. Then 
			\[ (S^{(\mathbf{\alpha})} v)_j= \frac{1}{h^{\alpha_j}} \sum_{ n=j-1}^{j-1+N} a_n^{(\alpha_j)} 
			\geq  -  \frac{1}{h^{\alpha_j}} \sum_{ n\geq N} a_n^{(\alpha_j)}  \geq  c_3   \frac{1}{h^{\alpha_j}} \cdot \frac{1}{N^{\alpha_j}} = c_3.
			\]
			By Lemma  \ref{lem:barrier}, we have  $ 	\|u\|_{L_h^\infty}=\|U\|_{l^\infty} \leq  c \|F\|_{l^\infty}= c\|f\|_{L_h^\infty}$. This completes proof.  
		\end{proof}

		Combining with the consistent error described in Theorem 	\ref{thm-approximation-property}, we can readily obtain the convergence.

		\section{Conclusion}\label{sec-conclusion}
		In this work, we considered the numerical evaluation of the variable-order fractional Laplacian particularly in 1D, 2D and 3D. 
		Based on the generating function theory on the  central finite difference schemes, we  derived an efficient finite difference method for the fractional Laplacian of the variable-order. 
  We developed a fast solver with the quasi-optimal complexity for the computation of discrete operators and numerical solution for the relevant nonlocal PDEs with variable-order fractional Laplacian. 
  We  discussed the implementation of the proposed scheme and reported numerical results to illustrate the accuracy and efficiency of our method and verify our theoretical predictions.   
 In this work,  we only provide the stability and convergence analysis in 1D and we hope to address the analysis in 2D and 3D in our future work.

	
		\section*{Acknowledgement}
		Z. Hao would like to thank Professor Yanzhi Zhang from Missouri University of Science and Technology for stimulating discussions for this work, and the support of a start-up grant from Southeast University in China.  R. Du would like to thank the support by the Natural Science Foundation of Jiangsu Province (Grant No. BK20221450).
		\bibliographystyle{siam}
		\bibliography{v_frac_lap_ref} 

\def\cprime{$'$} \def\cprime{$'$}
\begin{thebibliography}{10}

\bibitem{AnisworthG17}
{\sc M.~Ainsworth and C.~Glusa}, {\em Aspects of an adaptive finite element
  method for the fractional {L}aplacian: a priori and a posteriori error
  estimates, efficient implementation and multigrid solver}, Comput. Methods
  Appl. Mech. Engrg., 327 (2017), pp.~4--35.

\bibitem{Antil-Sinc-2021}
{\sc H.~Antil, P.~Dondl, and L.~Striet}, {\em Approximation of integral
  fractional {L}aplacian and fractional {PDE}s via sinc-basis}, SIAM J. Sci.
  Comput., 43 (2021), pp.~A2897--A2922.

\bibitem{Antil-2023-Sinc}
{\sc H.~Antil, P.~W. Dondl, and L.~Striet}, {\em Analysis of a sinc-galerkin
  method for the fractional laplacian}, SIAM J. Numer. Anal., 61 (2023),
  pp.~2967--2993.

\bibitem{Antil-Rautenberg-2019}
{\sc H.~Antil and C.~N. Rautenberg}, {\em Sobolev spaces with non-{M}uckenhoupt
  weights, fractional elliptic operators, and applications}, SIAM J. Math.
  Anal., 51 (2019), pp.~2479--2503.

\bibitem{bucur2016nonlocal}
{\sc C.~Bucur and E.~Valdinoci}, {\em Nonlocal diffusion and applications},
  Springer, 2016.

\bibitem{Burkard-Wu-Zhang-2021}
{\sc J.~Burkardt, Y.~Wu, and Y.~Zhang}, {\em A unified meshfree pseudospectral
  method for solving both classical and fractional {PDE}s}, SIAM J. Sci.
  Comput., 43 (2021), pp.~A1389--A1411.

\bibitem{Darve-DElia-Garrappa-2022}
{\sc E.~Darve, M.~D'Elia, R.~Garrappa, A.~Giusti, and N.~L. Rubio}, {\em On the
  fractional {L}aplacian of variable order}, Fract. Calc. Appl. Anal., 25
  (2022), pp.~15--28.

\bibitem{Marta-Glusa-2022}
{\sc M.~D'Elia and C.~Glusa}, {\em A fractional model for anomalous diffusion
  with increased variability: Analysis, algorithms and applications to
  interface problems}, Numer. Methods Partial Differential Equations, 38
  (2022), pp.~2084--2103.

\bibitem{Duo-Zhang-2019}
{\sc S.~Duo and Y.~Zhang}, {\em Accurate numerical methods for two and three
  dimensional integral fractional {L}aplacian with applications}, Comput.
  Methods Appl. Mech. Engrg., 355 (2019), pp.~639--662.

\bibitem{EMW2022}
{\sc W.~E, C.~Ma, and L.~Wu}, {\em The {B}arron space and the flow-induced
  function spaces for neural network models}, Constr. Approx., 55 (2022),
  pp.~369--406.

\bibitem{farquhar2018computational}
{\sc M.~E. Farquhar, T.~J. Moroney, Q.~Yang, I.~W. Turner, and K.~Burrage},
  {\em Computational modelling of cardiac ischaemia using a variable-order
  fractional {L}aplacian}, arXiv preprint arXiv:1809.07936,  (2018).

\bibitem{Feist-Bebendorf-2022}
{\sc B.~Feist and M.~Bebendorf}, {\em Quadrature rules for singular double
  integrals in {3D}}, arXiv:2208.05714,  (2022).

\bibitem{Felsinger-KV-2015}
{\sc M.~Felsinger, M.~Kassmann, and P.~Voigt}, {\em The {D}irichlet problem for
  nonlocal operators}, Math. Z., 279 (2015), pp.~779--809.

\bibitem{Fukushima-Uemura-2012}
{\sc M.~Fukushima and T.~Uemura}, {\em Jump-type {H}unt processes generated by
  lower bounded semi-{D}irichlet forms}, Ann. Probab., 40 (2012), pp.~858--889.

\bibitem{Gatto-Hesthaven2015}
{\sc P.~Gatto and J.~S. Hesthaven}, {\em Numerical approximation of the
  fractional {L}aplacian via {$hp$}-finite elements, with an application to
  image denoising}, J. Sci. Comput., 65 (2015), pp.~249--270.

\bibitem{Giusti-2020}
{\sc A.~Giusti, R.~Garrappa, and G.~Vachon}, {\em On the {K}uzmin model in
  fractional {N}ewtonian gravity}, Eur. Phys. J. Plus, 135 (2020).

\bibitem{gradshteyn2014table}
{\sc I.~S. Gradshteyn and I.~M. Ryzhik}, {\em Table of integrals, series, and
  products}, Academic press, 2014.

\bibitem{Gunzburger-18}
{\sc M.~Gunzburger, N.~Jiang, and F.~Xu}, {\em Analysis and approximation of a
  fractional {L}aplacian-based closure model for turbulent flows and its
  connection to {R}ichardson pair dispersion}, Comput. Math. Appl., 75 (2018),
  pp.~1973--2001.

\bibitem{Hao-Li-Zhang-Zhang-2021}
{\sc Z.~Hao, H.~Li, Z.~Zhang, and Z.~Zhang}, {\em Sharp error estimates of a
  spectral {G}alerkin method for a diffusion-reaction equation with integral
  fractional {L}aplacian on a disk}, Math. Comp., 90 (2021), pp.~2107--2135.

\bibitem{Hao-Zhang-Du-2021}
{\sc Z.~Hao, Z.~Zhang, and R.~Du}, {\em Fractional centered difference scheme
  for high-dimensional integral fractional {L}aplacian}, J. Comput. Phys., 424
  (2021), p.~109851.

\bibitem{huang2016finite}
{\sc Y.~Huang and A.~Oberman}, {\em Finite difference methods for fractional
  laplacians}, arXiv preprint arXiv:1611.00164,  (2016).

\bibitem{KLOKOV198793}
{\sc Y.~Klokov and A.~Shkerstena}, {\em Error estimates for lagrange-chebyshev
  interpolation formulae}, USSR Computational Mathematics and Mathematical
  Physics, 27 (1987), pp.~93--95.

\bibitem{Laskin2000}
{\sc N.~Laskin}, {\em Fractional quantum mechanics and {L}\'{e}vy path
  integrals}, Phys. Lett. A, 268 (2000), pp.~298--305.

\bibitem{lenzi2016anomalous}
{\sc E.~Lenzi, H.~Ribeiro, A.~Tateishi, R.~Zola, and L.~Evangelista}, {\em
  Anomalous diffusion and transport in heterogeneous systems separated by a
  membrane}, Proceedings of the Royal Society A: Mathematical, Physical and
  Engineering Sciences, 472 (2016), p.~20160502.

\bibitem{Liu-Sun-Zhang-2019}
{\sc X.~Liu, H.~Sun, Y.~Zhang, C.~Zheng, and Z.~Yu}, {\em Simulating
  multi-dimensional anomalous diffusion in nonstationary media using
  variable-order vector fractional-derivative models with kansa solver},
  Advances in Water Resources, 133 (2019), p.~103423.

\bibitem{MengM2022}
{\sc Y.~Meng and P.~Ming}, {\em A new function space from {B}arron class and
  application to neural network approximation}, Commun. Comput. Phys., 32
  (2022), pp.~1361--1400.

\bibitem{Minden-Ying-2020}
{\sc V.~Minden and L.~Ying}, {\em A simple solver for the fractional
  {L}aplacian in multiple dimensions}, SIAM J. Sci. Comput., 42 (2020),
  pp.~A878--A900.

\bibitem{MuHWZ2021}
{\sc X.~Mu, J.~Huang, L.~Wen, and S.~Zhuang}, {\em Modeling viscoacoustic wave
  propagation using a new spatial variable-order fractional laplacian wave
  equation}, Geophysics, 86 (2021), p.~T487–T507.

\bibitem{Ok2023}
{\sc J.~Ok}, {\em Local {H}\"{o}lder regularity for nonlocal equations with
  variable powers}, Calc. Var. Partial Differential Equations, 62 (2023),
  pp.~1--31.

\bibitem{Pang-Sun-2021}
{\sc H.-K. Pang and H.-W. Sun}, {\em A fast algorithm for the variable-order
  spatial fractional advection-diffusion equation}, J. Sci. Comput., 87 (2021),
  pp.~15--28.

\bibitem{Pinsky-2002}
{\sc M.~A. Pinsky}, {\em Introduction to {F}ourier analysis and wavelets},
  vol.~102 of Graduate Studies in Mathematics, American Mathematical Society,
  2009.

\bibitem{Ruiz-MedinaAA2004}
{\sc M.~D. Ruiz-Medina, V.~V. Anh, and J.~M. Angulo}, {\em Fractional
  generalized random fields of variable order}, Stochastic Analysis and
  Applications, 22 (2004), pp.~775--799.

\bibitem{Sheng-2020-SINUM}
{\sc C.~Sheng, J.~Shen, T.~Tang, L.-L. Wang, and H.~Yuan}, {\em Fast
  {F}ourier-like mapped {C}hebyshev spectral-{G}alerkin methods for {PDE}s with
  integral fractional {L}aplacian in unbounded domains}, SIAM J. Numer. Anal.,
  58 (2020), pp.~2435--2464.

\bibitem{Sheng-2023}
{\sc C.~Sheng, L.~Wang, H.~C. Chen, and H.~Li}, {\em Fast implementation of
  {FEM} for integral fractional {L}aplacian on rectangular meshes}, Commun.
  Comput. Phys.,  (2023).

\bibitem{Silvestre-2005}
{\sc L.~Silvestre}, {\em H\"{o}lder estimates for solutions of
  integro-differential equations like the fractional {L}aplace}, Indiana Univ.
  Math. J., 55 (2006), pp.~1155--1174.

\bibitem{Sun-Gao-2020}
{\sc Z.-z. Sun and G.-h. Gao}, {\em Fractional differential equations--finite
  difference methods}, De Gruyter, Berlin $\&$ Science Press, Beijing, 2020.

\bibitem{Tang-Wang-Yuan-Zhou-2020}
{\sc T.~Tang, L.-L. Wang, H.~Yuan, and T.~Zhou}, {\em Rational spectral methods
  for {PDE}s involving fractional {L}aplacian in unbounded domains}, SIAM J.
  Sci. Comput., 42 (2020), pp.~A585--A611.

\bibitem{tang2018hermite}
{\sc T.~Tang, H.~Yuan, and T.~Zhou}, {\em Hermite spectral collocation methods
  for fractional {PDE}s in unbounded domains}, Commun. Comput. Phys., 24
  (2018), pp.~1143--1168.

\bibitem{VdV92}
{\sc H.~A. van~der Vorst}, {\em {Bi-CGSTAB}: A fast and smoothly converging
  variant of bi-cg for the solution of nonsymmetric linear systems}, SIAM J.
  Sci. Stat. Comp., 13 (1992), pp.~631--644.

\bibitem{WangSLL2023}
{\sc Q.-Y. Wang, Z.-H. She, C.-X. Lao, and F.-R. Lin}, {\em Fractional centered
  difference schemes and banded preconditioners for nonlinear riesz space
  variable-order fractional diffusion equations}, Numer. Algorithms,  (2023).

\bibitem{Wang-Hao-Du-2022}
{\sc Y.~Wang, Z.~Hao, and R.~Du}, {\em A linear finite difference scheme for
  the two-dimensional nonlinear {S}chr\"{o}dinger equation with fractional
  {L}aplacian}, J. Sci. Comput., 90 (2022), pp.~Paper No. 24, 27.

\bibitem{Woyczynski2001}
{\sc W.~A. Woyczy{\'{n}}ski}, {\em L{\'e}vy Processes in the Physical
  Sciences}, Birkh{\"a}user Boston, Boston, MA, 2001, pp.~241--266.

\bibitem{Xu-Darve-2020}
{\sc K.~Xu and E.~Darve}, {\em Isogeometric collocation method for the
  fractional {L}aplacian in the 2{D} bounded domain}, Comput. Methods Appl.
  Mech. Engrg., 364 (2020), p.~112936.

\bibitem{YuZZZ-2022}
{\sc B.~Yu, X.~Zheng, P.~Zhang, and L.~Zhang}, {\em Computing solution
  landscape of nonlinear space-fractional problems via fast approximation
  algorithm}, J. Comput. Phys., 468 (2022), p.~111513.

\bibitem{Zheng-Wang-2020}
{\sc X.~Zheng and H.~Wang}, {\em An optimal-order numerical approximation to
  variable-order space-fractional diffusion equations on uniform or graded
  meshes}, SIAM J. Numer. Anal., 58 (2020), pp.~330--352.

\bibitem{ZhuangLAT2009}
{\sc P.~Zhuang, F.~Liu, V.~Anh, and I.~Turner}, {\em Numerical methods for the
  variable-order fractional advection-diffusion equation with a nonlinear
  source term}, SIAM J. Numer. Anal., 47 (2009), pp.~1760--1781.

\end{thebibliography}

  \appendix


		\section{
			Calculation of variable-order fractional Laplacian of the Gaussian function }
	We follow the similar argument for the constant case  as in   \cite{Sheng-2020-SINUM} to calculate the variable-order fractional Laplacian of the Gaussian function. Note that
		$$
		\begin{aligned}
			\mathcal{F}\left\{e^{-|x|^2}\right\}(\xi) & =\frac{1}{(2 \pi)^{d / 2}} \int_{\mathbb{R}^d} e^{-|x|^2} e^{-\mathbf{i} x \cdot \xi} \text{d} x 
   =\frac{1}{2^{d / 2}} e^{-\frac{|\xi|^2}{4}},
		\end{aligned}
		$$
		where we used the identity (cf. \cite{gradshteyn2014table}, P. 339]):
		$$
		\int_{\mathbb{R}} e^{-x^2} e^{-\mathrm{i} x \xi} \mathrm{d} x=\sqrt{\pi} e^{-\frac{\xi^2}{4}}.
		$$	
		Thus from the equivalence definition \eqref{def-v-frac-lap}, we obtain
		\begin{equation}
			\begin{aligned}
				(-\Delta)^{\alpha(x)/2}\left\{e^{-|x|^2}\right\}(x) & =\mathcal{F}^{-1}\left\{|\xi|^{\alpha(x)}\mathcal{F}\{e^{-|x|^2}\}(\xi)\right\} = \frac{1}{2^{d/2}(2\pi)^{d/2}}\int_{\mathbb{R}}|\xi|^{\alpha(x)}e^{-\frac{|\xi|^2}{4}}e^{\mathbf{i}x\xi}\text{d}\xi\\
				& =\frac{2^d}{2^{d / 2}(2 \pi)^{d / 2}} \int_{\mathbb{R}_{+}^d}|\xi|^{ \alpha(x)} e^{-\frac{|\xi|^2}{4}} \cos \left(x_1 \xi_1\right) \cos \left(x_2 \xi_2\right) \cdots \cos \left(x_d \xi_d\right) \mathrm{d} \xi,
			\end{aligned}
		\end{equation}
		We proceed with the calculation by using the $d$-dimensional spherical coordinates:
		\begin{eqnarray}
			\begin{aligned}
				& \xi_1=r \cos \theta_1 ; \xi_2=r \sin \theta_1 \cos \theta_2 ; \cdots \cdots ; \xi_{d-1}=r \sin \theta_1 \cdots \sin \theta_{d-2} \cos \theta_{d-1}; \\
				& \xi_d=r \sin \theta_1 \cdots \sin \theta_{d-2} \sin \theta_{d-1}, \quad r=|\xi|,
			\end{aligned}
		\end{eqnarray}
	Then we can write
		\begin{eqnarray}
			(-\Delta)^{\alpha(x)/2}\left\{e^{-|x|^2}\right\}(x)=\frac{1}{\pi^{d / 2}} \int_0^{\infty} r^{\alpha(x)+d-1} e^{-\frac{r^2}{4}} \mathcal{I}(r ; x) \mathrm{d} r,
		\end{eqnarray}
		where
		$$
		\begin{gathered}
			\mathcal{I}(r ; x)=\int_{\left[0, \frac{\pi}{2}\right]^{d-1}} \cos \left(r x_1 \cos \theta_1\right) \cos \left(r x_2 \sin \theta_1 \cos \theta_2\right) \cdots \cos \left(r x_{d-1} \sin \theta_1 \cdots \sin \theta_{d-2} \cos \theta_{d-1}\right) \\
			\cos \left(r x_d \sin \theta_1 \cdots \sin \theta_{d-2} \sin \theta_{d-1}\right)\left(\sin \theta_1\right)^{d-2}\left(\sin \theta_2\right)^{d-3} \cdots\left(\sin \theta_{d-2}\right) \mathrm{d} \theta_1 \mathrm{~d} \theta_2 \cdots \mathrm{d} \theta_{d-1} .
		\end{gathered}
		$$	
		We first integrate $\mathcal{I}(r ; x)$ with respect to $\theta_{d-1}$. To do this, we recall the integral formula involving the Bessel functions (cf. \cite{gradshteyn2014table}, P. 732]): for real $\mu, \nu>-1$ and $a, b>0$,
		\begin{eqnarray}
			\int_0^{\frac{\pi}{2}} J_\nu(a \sin \theta) J_\mu(b \cos \theta) \sin ^{\nu+1} \theta \cos ^{\mu+1} \theta \mathrm{d} \theta=\frac{a^\nu b^\mu J_{\nu+\mu+1}\left(\sqrt{a^2+b^2}\right)}{\left(a^2+b^2\right)^{(\nu+\mu+1) / 2}}.
		\end{eqnarray}	
		Then using the identity $\cos z=\sqrt{\pi z / 2} J_{-1 / 2}(z)$ and (A.4) (with $a=r x_{d-1} \sin \theta_1 \cdots \sin \theta_{d-2}, b=$ $r x_d \sin \theta_1 \cdots \sin \theta_{d-2}$ and $\left.\mu=\nu=-1 / 2\right)$, we derive
		$$
		\begin{gathered}
			\int_0^{\frac{\pi}{2}} \cos \left(r x_{d-1} \sin \theta_1 \cdots \sin \theta_{d-2} \cos \theta_{d-1}\right) \cos \left(r x_d \sin \theta_1 \cdots \sin \theta_{d-2} \sin \theta_{d-1}\right) \mathrm{d} \theta_{d-1} \\
			=\frac{\pi}{2} J_0\left(r \sin \theta_1 \cdots \sin \theta_{d-2} \sqrt{x_{d-1}^2+x_d^2}\right) .
		\end{gathered}
		$$
		Substituting the above into $\mathcal{I}(r, x)$, and applying the same argument to $\theta_{d-2}, \theta_{d-3}, \cdots, \theta_1$ iteratively $d-2$ times, we obtain
		\begin{eqnarray}
			\mathcal{I}(r ; x)=\left(\frac{\pi}{2}\right)^{\frac{d}{2}}(r|x|)^{1-\frac{d}{2}} J_{\frac{d}{2}-1}(r|x|) .
		\end{eqnarray}	
		We proceed with the integral identity (cf. \cite{gradshteyn2014table}, P. 713]): for real $\mu+\nu>-1$ and $p>0$,
		\begin{eqnarray}
			\int_{\mathbb{R}^{+}} J_\mu(b t) e^{-p^2 t^2} t^{\nu-1} \mathrm{~d} t=\frac{b^\mu \Gamma((\mu+\nu) / 2)}{2^{\mu+1} p^{\nu+\mu} \Gamma(\mu+1)}{ }_1 F_1\left(\frac{\mu+\nu}{2} ; \mu+1 ;-\frac{b^2}{4 p^2}\right) .
		\end{eqnarray}	
		Then, substituting (A.5) into (A.3) and using (A.6) (with $\mu=d / 2-1$ and $\nu=\alpha(x)+d / 2+1$ ), we derive
		$$
		\begin{aligned}
			(-\Delta)^{\alpha(x)/2}\left\{e^{-|x|^2}\right\}&=\frac{|x|^{1-\frac{d}{2}}}{2^{d / 2}} \int_0^{\infty} r^{\alpha(x)+\frac{d}{2}} e^{-\frac{r^2}{4}} J_{\frac{d}{2}-1}(r|x|) \mathrm{d} r\\
			&=\frac{2^{\alpha(x)} \Gamma((\alpha(x)+d) / 2)}{\Gamma(d / 2)}{ }_1 F_1\left(\frac{\alpha(x)+d}{2} ; \frac{d}{2} ;-|x|^2\right) . 
		\end{aligned}
		$$	
		Then \eqref{exact_solution} follows. This completes the proof.

	\end{document}